\documentclass[11pt]{article}
\usepackage[top=2.5cm, bottom=2.5cm, left=2.5cm, right=2.5cm] {geometry}
\usepackage{amsmath,amssymb,theorem,color}
\numberwithin{equation}{section} 
\usepackage{amssymb}
\usepackage{color}
\usepackage{amsmath}
\usepackage{graphicx}
\usepackage{amsfonts}
\newcommand{\R}{\ensuremath{\mathbb{R}}}

\newcommand{\N}{\ensuremath{\mathbb{N}}}
\newcommand{\eps}{\varepsilon}

\newcommand{\cC}{\mathcal{C}}
\newcommand{\cD}{\mathcal{D}}

\newcommand{\bP}{\mathbb{P}}

\newcommand{\mP}{\mathbb{P}}
\newcommand{\Z}{\mathbb{Z}}

\newcommand{\ltn}{\ensuremath{\left| \! \left| \! \left|}}
\newcommand{\rtn}{\ensuremath{\right| \! \right| \! \right|}}

\newtheorem{theorem}{Theorem}[section]
{ \theorembodyfont{\normalfont} 
	
	\newtheorem{remark}[theorem]{Remark}
}

\newtheorem{definition}[theorem]{Definition}
\newtheorem{lemma}[theorem]{Lemma}
\newtheorem{corollary}[theorem]{Corollary}
\newtheorem{proposition}[theorem]{Proposition}

\newcounter{enumctr}

\usepackage[displaymath, mathlines]{lineno}
\begin{document}

\title{Pullback attractors for stochastic Young differential delay equations}
\author{
Nguyen Dinh Cong\thanks{Institute of Mathematics, Vietnam Academy of Science and Technology, Vietnam {\it E-mail: ndcong@math.ac.vn}}, $\;$ Luu Hoang Duc\thanks{Max-Planck-Institut f\"ur Mathematik in den Naturwissenschaften, Leipzig, Germany \& Institute of Mathematics, Vietnam Academy of Science and Technology, Vietnam {\it E-mail: duc.luu@mis.mpg.de, lhduc@math.ac.vn}}, $\;$ Phan Thanh Hong \thanks{Thang Long University, Hanoi, Vietnam {\it E-mail: hongpt@thanglong.edu.vn }}\\[2ex]{\it in memory of Russell Johnson}}
\date{}
\maketitle
\begin{abstract}
	We study the asymptotic dynamics of stochastic Young differential delay equations under the regular assumptions on Lipschitz continuity of the coefficient functions. Our main results show that, if there is a linear part in the drift term which has no delay factor and has eigenvalues of negative real parts, then the generated random dynamical system possesses a random pullback attractor provided that the Lipschitz coefficients of the remaining parts are small.
\end{abstract}


\section{Introduction}
Consider the stochastic differential delay equation of the form
\begin{equation}\label{SDDE}
dy(t) = [Ay(t)+ f(y_t)]dt + g(y_t)dZ(t),\quad y_0=\eta\in \cC^{0,\beta_0}([-r,0],\R^d) \subset \cC_r: = \cC([-r,0],\R^d),
\end{equation}
where $t\in\R^+$, $y_t$ is defined by $y_t:[-r,0]\to\R^d$, $y_t(s)= y(t+s)$ for $s\in [-r,0]$, $A \in \R^{d \times d}$ is a matrix, $r$ is a constant delay, $\cC_r: = \cC([-r,0],\R^d)$ is the space of continuous functions on $[-r,0]$ valued in $\R^d$,  $f$ and $g$ are functions defined on $\cC_r$ valued in $\R^d$ and $\R^{d\times m}$ respectively, and $Z$ is a $\R^m$-valued  stochastic process with stationary increments on a probability space $(\Omega, \mathcal{F},\mathbb{P})$ which has almost sure all the realizations in the H\"older space $\cC^{0,\nu}$ for $\frac{1}{2}< \nu \leq 1$,   the initial condition belongs to the H\"older space $\cC^{0,\beta_0}([-r,0],\R^d)$. Equation \eqref{SDDE} is understood in the path-wise sense using Young integration \cite{young} for the stochastic term $g(y_t)dZ(t)$, whereas the term $[Ay(t)+ f(y_t)]dt$ is defined by the classical Riemann-Stieltjes integration. For the notion of Young integral and its properties, as well as notions and properties of spaces of H\"older continuous functions and H\"older norms  the reader is referred to Section~\ref{appendix} Appendix.

In this paper, we investigate the asymptotic behavior of solution of the delay system \eqref{SDDE} under regular assumptions. Namely,
\begin{itemize}
	\item  ${\bf H_1}$: $A$ has all eigenvalues of negative real parts;
	\item  ${\bf H_2}$: $f$ is globally Lipschitz continuous and thus has linear growth, i.e  there exists constants $C_f$ such that for all $\xi,\eta\in \cC_r$
	\[
	\|f(\xi)-f(\eta)\|\leq C_f\|\xi-\eta\|_{\infty,[-r,0]}; 
	\]
	\item ${\bf H_3}$: $g$ is $C^1$ such that its Frechet derivative is bounded and locally Lipschitz continuous, 
	i.e.\ there exists $C_g$ such that for all $\xi,\eta\in \cC_r$
	\[
	\|Dg(\xi)\|_{L(\cC_r,\R^d)}\leq C_g,
	\]
	and for each $M>0$, there exists $L_M$ such that for all $\xi,\eta\in \cC_r$ satisfying
	\[
	\|\xi\|_{\infty,[-r,0]},\|\eta\|_{\infty,[-r,0]}\leq M
	\] 
	one has
	\begin{equation}\label{Holderiv}
	\|Dg(\xi)- Dg(\eta)\|_{L(\cC_r,\R^d)}\leq L_M\|\xi-\eta\|_{\infty,[-r,0]}.
	\end{equation}
\end{itemize}
Notice that the same question for non-delay Young differential equations is well-studied in \cite{GAKLBSch2010}, \cite{ducGANSch18}, \cite{duchongcong18}, \cite{duchong19}, where one can prove that the system generates a random dynamical system which possesses a random attractor. For the delay system \eqref{SDDE}, the existence and uniqueness of the solution and the generation of a random dynamical system is affirmed in  \cite{BoufoussiHajji17}, \cite{ducetal15} and \cite{duchong18}, but the question on asymptotic stability is still open. 
  
Our aim in this paper is to show that under the assumptions ${\bf H_1,H_2,H_3}$, the system \eqref{SDDE} will generate a random dynamical system by means of its solution flow, and furthermore it possesses a random pullback attractor if the nonlinear term and stochastic term are small. Specifically, Theorem~\ref{attractor} states that if all the eigenvalues of $A$ have negative real parts ($\bf H_1$ holds) then, provided that the Lipschitz coefficients $C_f, C_g$ of the (perturbation) terms $f$ and $g$ are small, the random dynamical system generated by the equation \eqref{SDDE} possesses a random  pullback attractor. Although the result seems natural, its proof is rather technical which employs recently developed methods on semigroups and greedy sequence of stopping times \cite[Definition 4.7]{cassetal}, \cite[Section 2.2]{congduchong17}. In addition, we prove in Theorem~\ref{thm.bound} that, in case $g$ is bounded the assumption on the parameter $C_g$ as well as on the supremum norm of $g$ can be neglected in proving the existence of attractor. Moreover, Theorem~\ref{lin.attractor} asserts that, in case $g$ is linear the attractor is a singleton which is simultaneously a random pullback and random forward attractor. 

This paper is organized as follows. We present in section \ref{solutionsec} a  recurrence formula for the solution of deterministic delayed equation, hence a formula for estimating growth rate of solutions to the equation \eqref{SDDE}. Section~\ref{sec.genRDS} presents the generation of a random dynamical system from the delay equation \eqref{SDDE}. In Section~\ref{sec.main}, we present our main results on existence of a random pullback attractor for the generated random dynamical system. In Section~\ref{appendix}, for convenience of the reader we present some notions and notations used throughout the paper, namely the notions of Young integrals, H\"older spaces, H\"older norms; two versions of Gronwall inequalities---discrete and continuous are also presented.

\section{A recurrence formula for solutions of deterministic delay equation}\label{solutionsec}
In this section we consider the deterministic equation
\begin{equation}\label{YDDE}
dy(t) = [Ay(t)+ f(y_t)]dt + g(y_t)dx(t),\quad y_0=\eta\in \cC^{0,\beta_0}([-r,0],\R^d),
\end{equation}
for some $1-\nu< \beta_0<\nu$, and $x$ belongs to the $\cC^{0,\nu}([0,T],\R^m)$ for all $T>0$. By assumption, almost all realizations of $Z$ belong to $\cC^{0,\nu}$, hence \eqref{YDDE} is a representative path-wise equation of  the stochastic equation \eqref{SDDE}. 

Due to \cite{duchong18}, under the assumptions ${\bf H_1,H_2,H_3}$, the system \eqref{YDDE} has unique solution which belongs to $\cC^{\beta_0}([-r,T],\R^d) \cap \cC^{\beta}([0,T],\R^d)$ for all $T>0$, for all $\beta_0<\beta\leq \nu$. 

From now on, we fix $\beta_0\in (1-\nu,\nu)$, $\beta \in (\beta_0,\nu)$ and put
\begin{eqnarray*}
K&:=&\frac{1}{1-2^{1-(\beta+\nu)}},\\
K_0&:=&\frac{1}{1-2^{1-(\beta_0+\nu)}}
\end{eqnarray*}
(see details of the constants in the appendix). The following proposition is recalled from \cite[Lemmas 17.1, 17.2]{duchong18}.
\begin{proposition}\label{esth1}
	Let $h$ be a Lipschitz continuous function on $\cC_r$ with Lipschitz coefficient $L$ then for each $y\in \cC^{\alpha}([a-r,b],\R^d)$, $0<\alpha\leq 1$, $0\leq a<b$, we have\\
	$(i)$ $\|h(y_{\cdot})\|_{\infty,[a,b]} \leq \|h(0)\| + L\|y\|_{\infty,[a-r,b]}$, here $0$ denotes the zero element of $\cC_r$,\\
	$(ii)$ $\ltn h(y_{\cdot})\rtn_{\alpha,[a,b]} \leq L\ltn y\rtn_{\alpha,[a-r,b]}$.
\end{proposition}
\medskip

Denote by $\Delta_n$ and $\Delta'_n$ the intervals $[nr,(n+1)r]$ and $[(n-1)r,(n+1)r]$, respectively.  For each $0<\alpha< 1$, we introduce the notation 
\[
\|h\|_{\alpha,[a,b]}:=\|h\|_{\infty,[a,b]}+(b-a)^\alpha\ltn h\rtn_{\alpha,[a,b]} .
\]
It is obvious that $\|\cdot\|_{\alpha,[a,b]}$ and $\|\cdot\|_{\infty,\alpha,[a,b]}$ are equivalent  norms on $\cC^{\alpha}([a,b],\R^d)$.
We also introduce the following notations: 
\begin{itemize}
	\item For real numbers $a_1,\ldots, a_n$ put $a_1\wedge \ldots \wedge a_n := \min \{a_1,\ldots, a_n\}$, and 
	$a_1\vee \ldots \vee a_n := \max \{a_1,\ldots, a_n\}$;
	\item $L_f:=\|A\|+C_f$ with $\|A\|$ being the norm of operator $A$, i.e $\|A\|:=\sup_{\|x\|=1}\frac{\|Ax\|}{\|x\|}$;
	\item $\kappa:=4L_fr +2$.
\end{itemize}

In Proposition~\ref{esty1} below we prove a recurrence formula for the norm of the solution of \eqref{YDDE} by using the continuous Gronwall lemma and the technique of greedy sequences of stopping times like those in \cite{duchong19} with a modification for $\beta$-H\"older norm which is an appropriate norm to deal with the delay system as explained in \cite{duchong18}.

\begin{proposition}\label{esty1}
	The solution $y$ of the equation \eqref{YDDE},
\[
	dy(t) =  [Ay(t)+f(y_t)]dt + g(y_t)dx(t),\quad y_0=\eta\in  \cC^{\beta_0}([-r,0],\R^d),
\]
	satisfies
	\begin{eqnarray}\label{ykk}
	\|y\|_{\beta,\Delta_{n}}
	&\leq&e^{4L_fr+\kappa N_{n}(x)} \left[\|y\|_{\beta,\Delta_{n-1}} +\left(4r\|f(0)\|+\frac{\|g(0)\|}{C_g}\right)\right] - \left(4r\|f(0)\|+\frac{\|g(0)\|}{C_g}\right)
	\end{eqnarray}
	for all $n\geq 1$, and $N_n(x)$ is estimated by
	\begin{eqnarray}\label{Nn}
	N_n(x)\leq 1+[2(K+1)C_gr^\nu]^{\frac{1}{\nu-\beta} } \ltn  x\rtn^{\frac{1}{\nu-\beta}}_{\nu,\Delta_n}  .
	\end{eqnarray}
\end{proposition}
\begin{proof}
	Given an interval $[a,b]$ with $r\leq a\leq b$, notice that $y\in C^\beta([a-r,b],\R^d)$ and 
	\[
	 \|y_u\|_{\infty,[-r,0]}  \leq  \|y_v\|_{\infty,[-r,0]} +(u-v)^\beta\ltn y\rtn_{\beta,[v,u]}, \quad \forall v\leq u.
	\]
	Thus for $a\leq s<t \leq b$, it follows from the Young-Loeve estimate \eqref{yloeve} (see Appendix) and Proposition \ref{esth1} that
	\allowdisplaybreaks
	\begin{eqnarray}\label{ytsr}
	&&\|y(t)-y(s)\| \notag\\
	& =&\left \|\int_s^t  \left[Ay(u)+f(y_u)\right]du + \int_s^t g(y_u) dx(u)\right\|\notag\\
	&\leq& \int_s^t  \left(L_f \|y_u\|_{\infty,[-r,0]} +\|f(0)\|\right)du + \left\|\int_s^t g(y_u) dx(u)\right\|\notag\\ 
	&\leq & \int_s^t \left(L_f\|y_s\|_{\infty,[-r,0]}+ L_f(u-s)^\beta \ltn y\rtn_{\beta,[s,u]}+\|f(0)\|\right)du \notag\\
	&&+ (t-s)^\nu\ltn x\rtn_{\nu,[s,t]} \left[  C_g \|y_s\|_{\infty,[-r,0]} +\|g(0)\|+ KC_g(t-s)^\beta\ltn y\rtn_{\beta,[s-r,t]}  \right]\notag\\
	&\leq & (t-s) \left(L_f\|y_s\|_{\infty,[-r,0]}+\|f(0)\|\right)+L_f \int_s^t(u-s)^\beta \ltn y\rtn_{\beta,[s,u]}du \notag\\
	&&+ (t-s)^\nu\ltn x\rtn_{\nu,[s,t]} \left[  C_g \|y_a\|_{\infty,[-r,0]}+\|g(0)\| + C_g(s-a)^\beta\ltn y\rtn_{\beta,[a,s]}+ KC_g(t-s)^\beta\ltn y\rtn_{\beta,[s-r,t]}  \right].\notag\\
	\end{eqnarray}
	As a result,
	\begin{eqnarray*}
		&&\frac{\|y(t)-y(s)\|}{(t-s)^{\beta}} \notag\\
		&\leq & (t-s)^{1-\beta} \left(L_f\|y_a\|_{\infty,[-r,0]}+L_f(s-a)^\beta\ltn y\rtn_{\beta,[a,s]}+\|f(0)\|\right)+ L_f\int_s^t \ltn y\rtn_{\beta,[s,u]}du \notag\\
		&&+ (t-s)^{\nu-\beta}\ltn x\rtn_{\nu,[s,t]} \left[   C_g \|y_a\|_{\infty,[-r,0]} +\|g(0)\|+ C_g(s-a)^\beta\ltn y\rtn_{\beta,[a,s]}+ KC_g(t-s)^\beta\ltn y\rtn_{\beta,[s-r,t]} \right]\\
		&\leq & L_f \int_a^b \ltn y\rtn_{\beta,[a,u]}du  +L_f(b-a)\ltn y\rtn_{\beta,[a,s]}+(b-a)^{1-\beta} \left(L_f\|y_a\|_{\infty,[-r,0]}+\|f(0)\|\right)  \notag\\
		&&+ (b-a)^{\nu-\beta}\ltn x\rtn_{\nu,[a,b]}\times\\
		&&\times \left[  C_g \|y_a\|_{\infty,[-r,0]} + \|g(0)\| +C_g(b-a)^\beta\ltn y\rtn_{\beta,[a,b]}+ KC_g(b-a)^\beta \left(\ltn y\rtn_{\beta,[a-r,a]}+\ltn y\rtn_{\beta,[a,b]}\right) \right]\\
		&\leq &  L_f\int_a^b \ltn y\rtn_{\beta,[a,u]}du  +L_f\int_a^b\ltn y\rtn_{\beta,[a,s]} du +(b-a)^{1-\beta} \left(L_f\|y_a\|_{\infty,[-r,0]}+\|f(0)\|\right) \notag\\
		&&+ (b-a)^{\nu-\beta}\ltn x\rtn_{\nu,[a,b]} \times\\
		&&\times \left[ C_g \|y_a\|_{\infty,[-r,0]} + \|g(0)\| + KC_g(b-a)^\beta\ltn y\rtn_{\beta,[a-r,a]} +(K+1)C_g  (b-a)^\beta\ltn y\rtn_{\beta,[a,b]}\right]\\
		&\leq & L_f \int_a^b \ltn y\rtn_{\beta,[a,u]}du  +L_f\int_a^b\ltn y\rtn_{\beta,[a,u]} du+(b-a)^{1-\beta} \left(L_f\|y_a\|_{\infty,[-r,0]}+\|f(0)\|\right)  \notag\\
		&&+ (b-a)^{\nu-\beta}\ltn x\rtn_{\nu,[a,b]} \times\\
		&&\times \left[ C_g \|y_a\|_{\infty,[-r,0]} + \|g(0)\| + KC_g(b-a)^\beta\ltn y\rtn_{\beta,[a-r,a]} +(K+1)C_g  (b-a)^\beta\ltn y\rtn_{\beta,[a,b]}\right]\\
		&\leq & 2L_f \int_a^b \ltn y\rtn_{\beta,[a,u]}du  +(b-a)^{1-\beta} \left(L_f\|y_a\|_{\infty,[-r,0]}+\|f(0)\|\right)  \notag\\
		&&+ (b-a)^{\nu-\beta}\ltn x\rtn_{\nu,[a,b]} \times\\
		&&\times \left[ C_g \|y_a\|_{\infty,[-r,0]} + \|g(0)\| + KC_g(b-a)^\beta\ltn y\rtn_{\beta,[a-r,a]} +(K+1)C_g  (b-a)^\beta\ltn y\rtn_{\beta,[a,b]}\right].
	\end{eqnarray*}
	In other words,
	\begin{eqnarray}\label{ybt}
	\ltn y\rtn_{\beta,[a,b]} &\leq&  2 L_f\int_a^b \ltn y\rtn_{\beta,[a,u]}du  +(b-a)^{1-\beta} \left(L_f\|y_a\|_{\infty,[-r,0]}+\|f(0)\|\right)  \notag\\
	&&+ (b-a)^{\nu-\beta}\ltn x\rtn_{\nu,[a,b]} \times\notag\\
	&&\times \left[ C_g \|y_a\|_{\infty,[-r,0]} + \|g(0)\| + KC_g(b-a)^\beta\ltn y\rtn_{\beta,[a-r,a]} +(K+1)C_g  (b-a)^\beta\ltn y\rtn_{\beta,[a,b]}\right]\notag\\
	\end{eqnarray}
	For any $n\geq 1$ fixed, notice that \eqref{ybt} holds for all $[a,b]\subset \Delta_n=[nr,(n+1)r]$. Assign $\mu  := \frac{1}{2(K+1)C_gr^\beta}$ and construct on $\Delta_n$ a greedy sequence of stopping times $t_i$ as follows
	\[
	t_0 = nr,\quad t_{i+1} = \sup \{t > t_i| (t-t_i)^{\nu-\beta}\ltn x\rtn_{\nu,[t_i,t]} \leq \mu\}\wedge [(n+1)r].
	\]
	Since $x \in C^{0,\nu\rm{-Hol}}([0,T],\R^m)$,
	\[
	\Big|\ltn  x\rtn_{\nu,[0,\tau]} - \ltn  x\rtn_{\nu,[0,\tau\pm h]}\Big| \leq \max \Big\{\ltn  x\rtn_{\nu,[\tau,\tau+h]}, \ltn  x\rtn_{\nu,[\tau-h,\tau]} \Big\} 
	\]
	where the right hand side tends to zero as $h \to 0^+$, thus the function $\tau \mapsto \tau^{\nu-\beta}\ltn  x\rtn_{\nu,[0,\tau]}$ is continuous due to the continuity of each component in $\tau$. Hence 
	\begin{eqnarray}\label{stoppingtime}
	(t_{i+1} - t_i)^{\nu-\beta}   \ltn  x\rtn_{\nu,[t_i,t_{i+1}]}& =& \mu,\ \forall  0\leq i\leq N_n(x)-2,\\
	(t_{i+1} - t_i)^{\nu-\beta}   \ltn  x\rtn_{\nu,[t_i,t_{i+1}]}& \leq& \mu,\  {\rm for} \; i= N_n( x)-1,
	\end{eqnarray}
	where
	\[
	N_n(x)=N(\Delta_n, x) := 1 + \max \{i: t_i < (n+1)r\}.
	\]
	We are going to show that for this counting function $N_n(x)$ the inequalities \eqref{ykk} and \eqref{Nn} are satisfied. Indeed, we first prove that $N_n(x) $ is bounded and find an upper bound for it. Choose $m =\frac{1}{\nu-\beta}>1$, one has
	\begin{eqnarray*}
		[N_n(x)-1]\mu^m&=& \sum_{i=0}^{N_n(x)-2} \left[  (t_{i+1}-t_i)^{\nu-\beta} \ltn  x\rtn_{\nu,[t_i,t_{i+1}]}\right]^m\\
		&\leq & \sum_{i=0}^{N_n(x)-2} (t_{i+1}-t_i)^{m(\nu-\beta)} \ltn  x\rtn^m_{\nu,[t_i,t_{i+1}]}\\
		&\leq &  \sum_{i=0}^{N_n(x)-2}(t_{i+1}-t_i) \ltn  x\rtn^m_{\nu,\Delta_n}\\
		&\leq &   r \ltn  x\rtn^{\frac{1}{\nu-\beta}}_{\nu,\Delta_n}< \infty.
	\end{eqnarray*}
	Hence,
	\begin{eqnarray*}
		N_n(x)\leq 1+\frac{r }{\mu^{\frac{1}{\nu-\beta}}}  \ltn  x\rtn^{\frac{1}{\nu-\beta}}_{\nu,\Delta_n} =1+[2(K+1)C_gr^\nu]^{\frac{1}{\nu-\beta} } \ltn  x\rtn^{\frac{1}{\nu-\beta}}_{\nu,\Delta_n}.
	\end{eqnarray*}
	Thus $N_n(x)$ is bounded and the inequality \eqref{Nn} is proved.
	
	By the construction, $t_{i+1}-t_i\leq r$ for $0\leq i\leq N_n(x)-1$, hence for all $[a,b]\subset [t_i,t_{i+1}]$ the inequality \eqref{ybt} leads to
	\begin{eqnarray*}
		\ltn y\rtn_{\beta,[a,b]}
		&\leq & 2 L_f\int_a^b \ltn y\rtn_{\beta,[a,u]}du  +(b-a)^{1-\beta} \left(L_f\|y_a\|_{\infty,[-r,0]}+\|f(0)\|\right)  \notag\\
		&&+ \frac{1}{2(K+1)C_gr^\beta} \left[ C_g \|y_a\|_{\infty,[-r,0]} + \|g(0)\| \right] + \frac{1}{2r^\beta}(b-a)^\beta\ltn y\rtn_{\beta,[a-r,a]} +\frac{1}{2}\ltn y\rtn_{\beta,[a,b]}.\notag\\
	\end{eqnarray*}
	Hence, for all $[a,b]\subset [t_i,t_{i+1}]$,
	\begin{eqnarray*}
		\ltn y\rtn_{\beta,[a,b]}
		&\leq & 4 L_f\int_a^b \ltn y\rtn_{\beta,[a,u]}du  +2(b-a)^{1-\beta} \left(L_f\|y_a\|_{\infty,[-r,0]}+\|f(0)\|\right)  \notag\\
		&&\hspace{2cm}+\; \frac{1}{(K+1)C_gr^\beta} \left(C_g \|y_a\|_{\infty,[-r,0]} + \|g(0)\| \right) + \ltn y\rtn_{\beta,[a-r,a]}. \\
	\end{eqnarray*}
	In particular, for any $t\in[t_i,t_{i+1}]$,
	\begin{eqnarray*}
		\ltn y\rtn_{\beta,[t_i,t]}
		&\leq & 4 L_f\int_{t_i}^t \ltn y\rtn_{\beta,[t_i,u]}du  +2(t-t_i)^{1-\beta} \left(L_f\|y_{t_i}\|_{\infty,[-r,0]}+\|f(0)\|\right)  \notag\\
		&&\hspace{2cm}+ \frac{1}{(K+1)C_gr^\beta} \left( C_g \|y_{t_i}\|_{\infty,[-r,0]} + \|g(0)\| \right) +\ltn y\rtn_{\beta,[t_i-r,t_i]}.
	\end{eqnarray*}
	Using the continuous Gronwall lemma \ref{contgronwall}, we obtain
	\begin{eqnarray*}
		&&\ltn y\rtn_{\beta,[t_i,t]}\leq \\
		&\leq &\left[2(t-t_i)^{1-\beta} \left(L_f\|y_{t_i}\|_{\infty,[-r,0]}+\|f(0)\|\right) + \frac{\left( C_g \|y_{t_i}\|_{\infty,[-r,0]} + \|g(0)\| \right)}{(K+1)C_gr^\beta}  + \ltn y\rtn_{\beta,[t_i-r,t_i]}\right]\\
		&&\times  \left[1+4 L_f\int_{t_i}^{t} e^{4L_f(t-u)}du \right]\\
		&\leq &e^{4L_f(t-t_i)}\left[2(t-t_i)^{1-\beta} \left(L_f\|y_{t_i}\|_{\infty,[-r,0]}+\|f(0)\|\right) + \frac{\left( C_g \|y_{t_i}\|_{\infty,[-r,0]} + \|g(0)\| \right)}{(K+1)C_gr^\beta}  +\ltn y\rtn_{\beta,[t_i-r,t_i]}\right]\\
		&\leq &e^{4L_f(t-t_i)}\left[2L_f(t-t_i)^{1-\beta}\|y_{t_i}\|_{\infty,[-r,0]} + \ltn y\rtn_{\beta,[t_i-r,t_i]}+\frac{ \|y_{t_i}\|_{\infty,[-r,0]} }{(K+1)r^\beta} \right]\\
		&&+e^{4L_f(t-t_i)} \left(2(t-t_i)^{1-\beta}\|f(0)\|+\frac{\|g(0)\|}{(K+1)C_gr^\beta}\right).
	\end{eqnarray*}
	In particular, for $t:=t_{i+1}$,
	\begin{eqnarray*}
		\ltn y\rtn_{\beta,[t_i,t_{i+1}]}
		&\leq &e^{4L_f(t_{i+1}-t_i)}\left[2L_f(t_{i+1}-t_i)^{1-\beta}\|y_{t_i}\|_{\infty,[-r,0]} + \ltn y\rtn_{\beta,[t_i-r,t_i]}+\frac{ \|y_{t_i}\|_{\infty,[-r,0]} }{(K+1)r^\beta} \right]\\
		&&+e^{4L_f(t_{i+1}-t_i)} \left(2(t_{i+1}-t_i)^{1-\beta}\|f(0)\|+\frac{\|g(0)\|}{(K+1)C_gr^\beta}\right).
	\end{eqnarray*}
	Taking into account the equality $\|y_{t_i}\|_{\infty,[-r,0]} + r^\beta \ltn y\rtn_{\beta,[t_i-r,t_i]} = \|y\|_{\beta,[t_i-r,t_i]}$ and due to $t_{i+1}-t_i \leq r$, we obtain
	\begin{eqnarray*}
		&&2r^\beta\ltn y\rtn_{\beta,[t_i,t_{i+1}]}\\
		&\leq &e^{4L_f(t_{i+1}-t_i)}\left[\left(4L_fr^\beta (t_{i+1}-t_i)^{1-\beta} +\frac{2}{K+1}\right)\vee 2   \right]    \|y\|_{\beta,[t_i-r,t_i]}\\
		&&+e^{4L_f(t_{i+1}-t_i)} \left(4r^\beta (t_{i+1}-t_i)^{1-\beta}  \|f(0)\|+\frac{2\|g(0)\|}{(K+1)C_g}\right)\\
		&\leq &e^{4L_f(t_{i+1}-t_i)}   (4L_fr+2 )     \|y\|_{\beta,[t_i-r,t_i]}  +e^{4L_f(t_{i+1}-t_i)} \left(4r\|f(0)\|+\frac{\|g(0)\|}{C_g}\right)\\
		&\leq &e^{4L_f(t_{i+1}-t_i)}\left(e^{4L_fr +2} -1\right)   \|y\|_{\beta,[t_i-r,t_i]}+e^{4L_f(t_{i+1}-t_i)} \left(4r\|f(0)\|+\frac{\|g(0)\|}{C_g}\right)\\
		&\leq &e^{4L_f(t_{i+1}-t_i)}\left(e^{\kappa} -1\right)     \|y\|_{\beta,[t_i-r,t_i]}+e^{4L_f(t_{i+1}-t_i)} \left(4r\|f(0)\|+\frac{\|g(0)\|}{C_g}\right)\\
		&\leq &e^{4L_f(t_{i+1}-t_i)+\kappa}    \|y\|_{\beta,[t_i-r,t_i]}- \|y\|_{\beta,[t_i-r,t_i]}+e^{4L_f(t_{i+1}-t_i)} \left(4r\|f(0)\|+\frac{\|g(0)\|}{C_g}\right).
	\end{eqnarray*}
	Now we evaluate norm of $y$ on $[t_{i+1}-r,t_{i+1}]$ as follows
	\begin{eqnarray}
	\|y\|_{\beta,[t_{i+1}-r,t_{i+1}]}&=& \|y\|_{\infty,[t_{i+1}-r,t_{i+1}]} + r^\beta \ltn y\rtn_{\beta,[t_{i+1}-r,t_{i+1}]}\notag\\
	&\leq &  \|y\|_{\infty,[t_{i}-r,t_i]} + (t_{i+1}-t_i)^\beta \ltn y\rtn_{\beta,[t_i,t_{i+1}]}+r^\beta( \ltn y\rtn_{\beta,[t_{i}-r,t_{i}]}+  \ltn y\rtn_{\beta,[t_{i},t_{i+1}]})\notag\\
	&\leq& \|y\|_{\beta,[t_i-r,t_i]} +2r^\beta  \ltn y\rtn_{\beta,[t_{i},t_{i+1}]}\notag\\
	&\leq&e^{4L_f(t_{i+1}-t_i)+\kappa} \|y\|_{\beta,[t_i-r,t_i]} +e^{4L_f(t_{i+1}-t_i)} \left(4r\|f(0)\|+\frac{\|g(0)\|}{C_g}\right),\notag
	\end{eqnarray}
	where, to derive the second inequality, we used the estimate 
	\[
	\| y\|_{\infty,[t_i,t_{i+1}]} \leq  \| y(t_i)\| +(t_{i+1}-t_i)^\beta\ltn y\rtn_{\beta,[t_{i},t_{i+1}]} \leq  \| y\|_{\infty,[t_{i}-r,t_{i}]} +(t_{i+1}-t_i)^\beta\ltn y\rtn_{\beta,[t_{i},t_{i+1}]}.
	\]
	By induction we obtain, for any $i=0,\dots, N_n(x)-1$,
	\begin{eqnarray*}
		&&\|y\|_{\beta,[t_{i+1}-r,t_{i+1}]}\notag\\
		&\leq&e^{4L_f(t_{i+1}-t_0)+(i+1)\kappa} \|y\|_{\beta,[t_0-r,t_0]} +\left(4r\|f(0)\|+\frac{\|g(0)\|}{C_g}\right) \sum_{j=0}^ie^{4L_f(t_{i+1}-t_j)+(i-j)\kappa} \notag\\
		&\leq&e^{4L_f(t_{i+1}-t_0)+(i+1)\kappa}  \left[\|y\|_{\beta,[t_0-r,t_0]} +\left(4r\|f(0)\|+\frac{\|g(0)\|}{C_g}\right) \right] - \left(4r\|f(0)\|+\frac{\|g(0)\|}{C_g}\right) .
	\end{eqnarray*}
	Replacing $i= N_n(x)-1$ with note that $[t_0-r,t_0] = \Delta_{n-1}$ , $[t_{N_n(x)}-r,t_{N_n(x)}] = \Delta_n$, we obtain
	\begin{eqnarray*}
		\|y\|_{\beta,\Delta_n}
		&\leq&e^{4L_fr+\kappa N_{n}(x)} \left[\|y\|_{\beta,\Delta_{n-1}} +\left(4r\|f(0)\|+\frac{\|g(0)\|}{C_g}\right)\right] - \left(4r\|f(0)\|+\frac{\|g(0)\|}{C_g}\right),
	\end{eqnarray*}
which proves \eqref{ykk}.
\end{proof}
\begin{remark}\label{specialk1}
	Notice that while the solution of \eqref{YDDE} belongs to $\cC^\beta$ on $[0,T]$, it only belongs to $\cC^{\beta_0}$ but not necessarily belongs to  $\cC^\beta$ on $[-r,0]$. Therefore we have to make separate estimations for solutions of \eqref{YDDE} on the first interval $[0,r]$. By a slight modification of the proof of Proposition~\ref{esty1} we obtain the following estimates.
	\begin{enumerate}
		\item It is evident that if we replace $\beta$ by $\beta_0$ and $K$ by $K_0$, then \eqref{ykk} and \eqref{Nn} hold for all $n\geq 0$. In particular, letting $n=0$ we have an estimate in the $\|\cdot\|_{\beta_0,[0,r]}$ norm for the solution of \eqref{YDDE} on $[0,r]$  as follows  
		\begin{eqnarray}\label{ykk0}
		\hspace{-1cm}	\|y\|_{\beta_0,[0,r]} \leq  e^{4L_fr+\kappa N_{0}(x)} \left[\|y\|_{\beta_0,[-r,0]} +\left(4r\|f(0)\|+\frac{\|g(0)\|}{C_g}\right)\right] 
		- \left(4r\|f(0)\|+\frac{\|g(0)\|}{C_g}\right),
		\end{eqnarray}
		where 
		\begin{eqnarray}\label{Nk0}
		N_0(x) \leq  1+[2(K_0+1)C_gr^\nu]^{\frac{1}{\nu-\beta_0} } \ltn  x\rtn^{\frac{1}{\nu-\beta_0}}_{\nu,[0,r]}.   
		\end{eqnarray}
		\item Similar to \eqref{ytsr}, for  $0\leq s<t\leq r$ we have
		\begin{eqnarray*}
			\|y(t)-y(s)\| 
			& =&\left \|\int_s^t  \left[Ay(u)+f(y_u)\right]du + \int_s^t g(y_u) dx(u)\right\|\notag\\
			&\leq& \int_s^t  \left(L_f \|y_u\|_{\infty,[-r,0]} +\|f(0)\|\right)du + \left\|\int_s^t g(y_u) dx(u)\right\|\notag\\
			&\leq& (t-s) \left(L_f\|y\|_{\infty,[-r,r]}+\|f(0)\|\right)\\
			&&+ (t-s)^\nu\ltn x\rtn_{\nu,[s,t]} \left[  C_g \|y\|_{\infty,[-r,r]} +\|g(0)\|+ K_0C_gr^{\beta_0}\ltn y\rtn_{\beta_0,[-r,r]}  \right]\\
			&\leq & D \left [(t-s)  + (t-s)^\nu\ltn x\rtn_{\nu,[s,t]} \right] \Big[1+\|y\|_{\infty,[-r,r]} + r^{\beta_0}\ltn y\rtn_{\beta_0,[-r,r]} \Big] \\
			&\leq & D(t-s)^\beta \left (r^{1-\beta}  + r^{\nu-\beta}\ltn x\rtn_{\nu,[0,r]} \right) \Big(1+\|y\|_{\beta_0,[-r,0]} + \| y\|_{\beta_0,[0,r]} \Big)
		\end{eqnarray*}
		for some positive constants $D$. Combining this with \eqref{ykk0} and changing the constant $D$ to a bigger one if necessary, we obtain the following  estimate in the $\|\cdot\|_{\beta,[0,r]}$ norm for the solution of \eqref{YDDE} on $[0,r]$  
		\begin{eqnarray}\label{ybeta0r}
		\| y\|_{\beta,[0,r]} 
		&\leq & D \left(1  +\ltn x\rtn_{\nu,[0,r]} \right) \Big[1+\|y\|_{\beta_0,[-r,0]} + \| y\|_{\beta_0,[0,r]} \Big] \notag\\
		&\leq & D \left(1  +\ltn x\rtn_{\nu,[0,r]} \right) \Big(1+\|y\|_{\beta_0,[-r,0]} \Big) e^{\kappa N_0(x)} \notag\\
		&\leq&D\left(1  +\ltn x\rtn_{\nu,[0,r]} \right) \Big(1+\|y\|_{\beta_0,[-r,0]} \Big)  e^{D \ltn  x\rtn^{\frac{1}{\nu-\beta_0}}_{\nu,[0,r]} }.
		\end{eqnarray}
	\end{enumerate}
\end{remark}

\section{Generation of random dynamical systems}\label{sec.genRDS}

In this section, we present the generation of random dynamical systems for equation \eqref{SDDE} under general noise $Z$, a stochastic process with stationary increments with almost sure all the realizations in $\cC^{0,\nu}$. Namely, similar to \eqref{SDDE} but for simplicity of presentation we consider the equation
\begin{equation}\label{SDDE2}
dy(t) = F(y_t)dt+g(y_t)dZ(t),\;\; y_0=\eta\in \cC^{0,\beta_0}([-r,0],\R^d),
\end{equation} 
where $\beta_0>1-\nu$ is an arbitrary fixed constant. 
Note that \eqref{SDDE} is a special case of \eqref{SDDE2} with the coefficient $F(y_t)$ changed to $Ay(t)+f(y_t)$. The initial condition is considered in the separable space $\cC^{0,\beta_0}([-r,0],\R^d)$, the condition $\beta_0>1-\nu$ is needed to assure existence and uniqueness of solution to \eqref{SDDE2} in the $\cC^{0,\beta_0}$ space (see \cite{duchong18}).
First, we recall the definition of {\it random dynamical system} (RDS). Let $(\Omega,\mathcal{F},\mP)$ be a probability space equipped with a so-called {\it metric dynamical system} $\theta$, which is a measurable mapping $\theta: \R \times \Omega \to \Omega$ such that $\theta_t:\Omega\to\Omega$ is $\mP-$ preserving, i.e.\ $\mP(B) = \mP(\theta^{-1}_t(B))$ for all $B\in \mathcal{F}, t\in \R$, and $\theta_{t+s} = \theta_t \circ \theta_s$ for all $t,s \in \R$.  Let $S$ be a Polish space, a  {\it continuous random dynamical system} 
\[
\varphi: \R \times \Omega \times S \to S,\; (t,\omega, y_0)\mapsto \varphi(t,\omega,y_0)
\]
 is then defined as a measurable mapping which is also continuous in $t$ and $y_0$ such that the cocycle property
\begin{eqnarray}\label{cocycle}
\varphi(t+s,\omega) &=&\varphi(t,\theta_s \omega) \circ \varphi(s,\omega),\quad \forall t,s \in \R, {\textcolor{black} \forall \omega\in \Omega}\\
\varphi(0,\omega,\cdot)&=& Id_S
\end{eqnarray}
is satisfied (see Arnold~\cite{arnold}).
  
To study the existence of the random pullback attractor of the system \eqref{SDDE} in the next section, we need to construct a canonical space for $Z$ which is equipped by a metric dynamical system $\theta$. In the following, we follow \cite[Theorem\ 5]{BRSch17}  to state a similar result for stochastic process valued in $\cC^{\alpha}$ for some $\alpha\in (0,1]$. 
Recall that $\cC^{0,\alpha}([a,b],\R^m)$ is the closure of smooth path from $[a,b]$ to $\R^m$ in $\alpha$-H\"older norm. It is known that $\cC^{0,\alpha}([a,b],\R^m)$ is a separable Banach space, see \cite{friz}. Denote by $\cC^{0,\alpha}(\R,\R^m)$ the space of all $x: \R\to \R^m$ such that $x|_I \in \cC^{0,\alpha}(I, \R^m)$ for each compact interval $I\subset\R$, equipped  with the compact open topology given by the $\alpha$-H\"older norm, i.e.\  topology generated by the metric:
\[
d(x,y): = \sum_{n\geq 1} \frac{1}{2^n} (\|x-y\|_{\infty,\alpha,[-n,n]}\wedge 1).
\]
Then $\cC^{0,\alpha}(\R,\R^m)$ is a separable metric space. Denote by $\cC^{0,\alpha}_0(\R,\R^m)$ the subspace of $\cC^{0,\alpha}(\R,\R^m)$ containing paths which vanish at $0$.
It is evident that for $x\in \cC^{0,\alpha}_0(\R,\R^m)$
\[
 \ltn x\rtn_{\alpha,[-n,n]}\leq \|x\|_{\infty,\alpha,[-n,n]} \leq (1+n^\alpha) \ltn x\rtn_{\alpha,[-n,n]}
\]
for all $n$, and  $\cC^{0,\alpha}_0(\R,\R^m)$ is closed in $\cC^{0,\alpha}(\R,\R^m)$.
The following Theorem is due to \cite{BRSch17}. 
\begin{theorem} 
Assume that we  have a process $\bar{X}$ defined on a probability space $(\bar{\Omega},\bar{\mathcal{F}},\bar{\bP})$ and  valued in $(\cC^{0,\alpha}_0(\R,\R^m), \mathcal{B})$ with $\mathcal{B}$ being Borel $\sigma$-algebra. Assume further that $\bar{X}$ has stationary increment. Then there exist a metric dynamical system $(\Omega,\mathcal{F},\bP, (\theta_t)_{t\in \R})$ and a process $X: (\Omega,\mathcal{F},\bP) \to \cC^{0,\alpha}_0(\R,\R^m)$ which has the same law as $\bar{X}$ and satisfies the property:
$$X(t+s,\omega) = X(s,\omega) + X(t,\theta_s\omega),\quad \forall  \omega \in \Omega,\quad t,s\in \R.$$
\end{theorem}
\begin{proof}
We denote by $(\Omega, \mathcal{F})$ the space  $(\cC^{0,\alpha}_0(\R,\R^d),\mathcal{B})$  and by $\bP$ the distribution of $\bar{X}$ on $\Omega$. On  $(\Omega,\mathcal{F},\bP)$ we set 
$$
\theta:\R\times \Omega  \to \Omega,\quad \theta(t,\omega)(s) = \theta_t\omega(s) := \omega(t+s) - \omega(t),
$$ 
and define the process $X$: $X(\omega) (t)= \omega(t)$ for all $\omega\in \Omega$. The properties of $(\Omega,\mathcal{F},\bP)$ and $X$ are obtained by arguments similar to that of \cite[Theorem~5, p.~8]{BRSch17}.
\end{proof}

\medskip
Now we consider the systems \eqref{SDDE} and  \eqref{SDDE2} with $Z$ defined on the canonical space constructed as above. Moreover, we assume that $(\Omega,\mathcal{F},\bP,\theta)$ is ergodic and 
\begin{equation}\label{Gamma}
\Gamma(\beta) : = \Big(E \ltn Z \rtn^{\frac{1}{\nu-\beta}}_{\beta,[-r,r]}\Big)^{\nu-\beta} < \infty.
\end{equation}
Next, we are going to study the generation of random dynamical system from the system \eqref{SDDE2}. Note that the Young integral satisfies the shift property with respect to $\theta$ (see for instance \cite{congduchong18}), i.e.
\[
\int_a^b x(u)d\omega(u) = \int_{a-r}^{b-r} x(u+r) d\theta_r \omega(u).  
\]
and due to \cite{duchong18} 
 the equation \eqref{SDDE2} possesses a unique solution $y(t,x,\eta)$ in $\cC^{0,\beta_0}([-r,\infty),\R^d)$. Moreover, the solution is continuous w.r.t $\eta$ and belongs to $\cC^{\beta}([0,\infty),\R^d)$ for $\beta_0<\beta<\nu$. The following conclusion is followed from \cite{ducetal15}.
\begin{theorem}\label{gen.rds}
	Under assumption ${\bf {(H_2), (H_3)}}$ the system \eqref{SDDE2} generates a random dynamical system defined by
	\[
	\varphi: \R^+\times\Omega\times \cC^{0,\beta_0}([-r,0],\R^d) \rightarrow \cC^{0,\beta_0}([-r,0],\R^d),\quad \varphi(t,x,\eta) (s): =y(t+s,x,\eta).
	\]
	Moreover, $\varphi$ is continuous.
\end{theorem}

\begin{corollary}
Under assumption ${\bf {(H_1), (H_2), (H_3)}}$ the stochastic delay equation \eqref{SDDE} generates a continuous random dynamical system with the phase space $\cC^{0,\beta_0}([-r,0],\R^d)$.
	\end{corollary}

\section{Random pullback attractors}\label{sec.main}
This section is devoted to the main result of our paper. We will show that under some natural conditions the random dynamical system generated by the stochastic Young differential delay equation \eqref{SDDE} possesses a random pullback attractor. Note that for the classical theory of dynamical systems one usually studies forward attractor, but in the framework of the theory of random dynamical systems the notion of random pullback attractor seems more appropriate (see e.g.\ \cite{crauelkloeden} and the references therein). The relation between concepts of attractors is studied in \cite{crauel}, \cite{crauelscheutzow}, \cite{flandolibjorn}, \cite{chebanetal}. Particularly in relation to the nonautonomous setting with compact topological parameter space, there is a work by \cite{Johnson2009} which proves that the (nonautonomous) pullback attractor of nonautonomous dynamical systems (in terms of skew product flows) coincides with their so-called  Lyapunov attractors.

First we recall the classical notion of random pullback attractors for a general random dynamical system.
Let $S$ be a Polish space, i.e.\ a separable topological space whose topology is metrizable with a complete metric $d$. Denote by $\mathcal{B}$ the Borel-$\sigma$ algebra on $S$. For each $y \in S$, $E\subset S$, we define $d(y,E)=\inf\{d(y,z)|\ z\in E\}$. The Hausdorff distance between two nonempty subsets $E,F$ of $S$ is defined by  $d(E|F) = \sup\{\inf\{d(y, z)|z \in F\} | y \in E\}$. Recall that a set $\hat{M} =\{M(x)\}_{x \in \Omega}$ is called a {\it random set} if it belongs to $\mathcal{F}\times\mathcal{B}$. In the case that $M(x)$ is closed or compact  for each  $x\in\Omega$, that the mapping $x\mapsto d(y,M(x))$ is measurable for every $y\in S$ ensures the measurability of $M$. $M$ is then said to be closed or compact random set.
Given a random dynamical system $\varphi$ on $(\Omega,\mathcal{F},\bP)$, valued on $S$. 
 We recall the following definition from  \cite{flandolibjorn, crauelscheutzow}.
\begin{definition}\label{dfn.attractor}
Suppose that $\varphi$ is a RDS on a Polish space $S$ and $\cD$ is a non-empty family of subsets of $\Omega\times S$. Then a set $\mathcal{A}\subset \Omega\times S$ is a {\em random pullback attractor for $\cD$} if\\
$(i)$ $\mathcal{A}$ is a compact random set,\\
$(ii)$ $\mathcal{A}$ is strictly $\varphi-$invariant, i.e.\ $\varphi(t,x)\mathcal{A}(x) =\mathcal{A}(\theta_t x)$ $\bP$-almost surely for every $t\geq 0$,\\
$(iii)$ $\mathcal{A}$ attracts $\cD$ in the pullback sense, i.e for every $\hat{D} \in \cD$,
\begin{equation}\label{pullback}
\lim \limits_{t \to \infty} d(\varphi(t,\theta_{-t}x)
\hat{D}(\theta_{-t}x)| \mathcal{A} (x)) = 0,\;\;\; \bP-a.s.
\end{equation}
\end{definition}

It is known that, under the assumption on the continuity with respect to the state variable of the random dynamical system,  the existence of the random pullback attractor follows from the existence of the random pullback absorbing set (see \cite[Theorem 3.5]{flandolibjorn}, \cite[Theorem 2.4]{GAKLBSch2010}), i.e.\ a  compact random set $B$ such that $\bP-$almost all $x$, for each $\hat{D}\in \cD$  there
exists a time $t_0(x, \hat{D})$ such that for all $t>t_0(x,\hat{D})$,
\[
\varphi(t,\theta_{-t}x)\hat{D}(\theta_{-t}x)\subset B(x).
\]
An {\it universe} $\cD$ is a family of random sets
which is closed w.r.t.\ inclusions (i.e. if $\hat{D}_1 \in \cD$ and
$\hat{D}_2 \subset \hat{D}_1$ then $\hat{D}_2 \in \cD$). Given a universe $\cD$ and a compact random
pullback absorbing set $\mathcal{B} \in \cD$, there exists a unique random pullback attractor
 in $\cD$, given by
\begin{equation}\label{at}
\mathcal{A}(x) = \cap_{s \geq 0} \overline{\cup_{t\geq s} \varphi(t,\theta_{-t}x)\mathcal{B}(\theta_{-t}x)}. 
\end{equation}
 
In our setting, the problem of generation of  random dynamical system by a stochastic Young differential delay equation is treated in Section~\ref{sec.genRDS} and our equation \eqref{SDDE} generated a continuous random dynamical system  with the phase space being the function space $\cC^{0,\beta_0}([-r,0],\R^d)$.
We define the universe $\cD$ to be a family of {\it tempered} random sets $D(x)$ which  is contained in a ball $B(0,\rho(x))$ a.s., where the radius $\rho(x)$ is a tempered random variable (see Appendix).

Here we notice that while the definition of pullback attractor is formulated for a general universe $\cD$, or even for the case $\cD$ being an abstract collection of subsets of the product of the phase space and the probability space of the random dynamical system 
(see \cite[Definition 9]{crauelscheutzow}, \cite[Definition 15]{crauelkloeden}), in practical concrete problems one needs to impose additional conditions on the growth rate of the size of the random sets $\hat{D}(\cdot)$. Thus one may consider the universe of tempered compact random sets (see \cite[Theorem 5.10]{flandolibjorn}, \cite[Theorem 2.4]{GAKLBSch2010}), or the universe of deterministic bounded sets (see  \cite[Definition 1.3]{scheutzow2002}). In this paper
we follow \cite{flandolibjorn,GAKLBSch2010} in imposing temperedness condition on the universe $\cD$ as above.

Now, to understand the dynamics of the random dynamical system generated by the stochastic Young differential delay equation \eqref{SDDE} we need to study the path-wise deterministic equation of \eqref{SDDE}. Let us look back at the system \eqref{YDDE}
\[
dy(t) = [Ay(t)+ f(y_t)]dt + g(y_t)dx(t), y_0=\eta\in \cC^{0,\beta_0}([-r,0],\R^d),
\]
with the assumptions $\bf {H_1, H_2,H_3}$. Put $\Phi(t) :=e^{At}$. By the variation of constants formula, the solution $y(t)$ of  \eqref{YDDE} satisfies 
 \begin{equation}\label{yvar}
 y(t)= \Phi(t-t_0)y(t_0) + \int_{t_0}^t\Phi(t-s)f(y_s)ds+\int_{t_0}^t\Phi(t-s)g(y_s)dx(s),\quad \forall t\geq t_0\geq 0.
 \end{equation}
 Indeed, put $z(t) = \Phi^{-1}(t)y(t)$ then
 \begin{eqnarray*}
 dz(t)& =& d\Phi^{-1}(t) y(t) + \Phi^{-1}(t)dy(t)\notag\\
 &=& -\Phi^{-1}(t)d\Phi(t)\Phi^{-1}(t)y(t) + \Phi^{-1}(t) \Big[ (Ay(t) + f(y_t))dt + g(y_t)dx(t)\Big] \notag\\
 &=&-\Phi^{-1}(t) Ay(t) dt + \Phi^{-1}(t) \Big[ (Ay(t) + f(y_t))dt + g(y_t)dx(t)\Big] \notag\\
 &=& \Phi^{-1}(t) \Big[f(y_t)dt + g(y_t)dx(t)\Big], \notag
 \end{eqnarray*}
 hence
 $y(t) =\Phi(t) z(t) = 
 \Phi(t) \Big(z(t_0)+ \int_{t_0}^t\Phi^{-1}(s)f(y_s)ds + \int_{t_0}^t\Phi^{-1}(s)g(y_s)dx(s)\Big),$
 from which  \eqref{yvar} follows.   
 
 By the assumption $\bf {H_1}$ on $A$, there exist  positive constants $C_A,\lambda  >0$  (see \cite[Chapter 1, \S3]{Adrianova})  such that 
	\begin{eqnarray}
	\|\Phi\|_{\infty,[a,b]} &\leq& C_Ae^{-\lambda  a}, \label{estphi1}\\
\ltn \Phi\rtn_{\alpha,[a,b]} &\leq& \|A\|C_Ae^{-\lambda  a} (b-a)^{1-\alpha}, \quad \forall\;  0< a<b, \label{estphi2}
\end{eqnarray}
where $0<\alpha\leq 1$ is arbitrary and fixed, and $\|A\|$ is the norm of operator $A$.
 
 We introduce the following notations
\begin{eqnarray}
\lambda_0&:=&\lambda-L, \quad\hbox{where}\quad L:=C_AC_fe^{\lambda r},\label{notationL}\\
M_1&:=&KC_Ae^{4\lambda r} r^\nu (1+\|A\|r).\label{notationM1}
\end{eqnarray}
From now on, we will assume that 
\begin{eqnarray}\label{lambda0_positive}
\lambda_0=\lambda-L >0.
\end{eqnarray}
The following lemma gives us an estimate for the uniform norm of solutions to the deterministic Young equation \eqref{YDDE}.
\begin{lemma}\label{lem.yinfty}
Let $n\in\N$, $n\geq 1$, and $t\in \Delta_n$ be arbitrary.  Then, there exists a positive constant $M_2$ independent of $t$ and $n$ such that
the solution of system \eqref{YDDE} satisfies
\begin{eqnarray}\label{yinfty}
\|y\|_{\infty,[t-r,t]}
 &\leq&M_2 e^{-\lambda_0 nr}  \|y\|_{\infty,[0,r]} + (\|f(0)\|\vee \|g(0)\|) M_2\sum_{k=0}^{n-1} (1+ \ltn x\rtn_{\nu,\Delta_{k+1}})e^{-\lambda_0 (n-k)r} \notag\\
&&\hspace{2cm}+\; C_gM_1 \sum_{k=0}^{n-1} \ltn x\rtn_{\nu,\Delta_{k+1}}e^{-\lambda_0 (n-k)r}\left( \|y\|_{\beta,\Delta_{k}}+ \|y\|_{\beta,\Delta_{k+1}}\right),
	\end{eqnarray}
where $M_1$ is defined by the formula \eqref{notationM1}. In particular, 
\begin{eqnarray}\label{yinftyn}
		\|y\|_{\infty,\Delta_n} &\leq& M_2 e^{-\lambda_0 nr}  \|y\|_{\infty,[0,r]} + M_2(\|f(0)\|\vee \|g(0)\|) \sum_{k=0}^{n-1} (1+ \ltn x\rtn_{\nu,\Delta_{k+1}})e^{-\lambda_0 (n-k)r} \notag\\
&&+\; C_gM_1 \sum_{k=0}^{n-1} \ltn x\rtn_{\nu,\Delta_{k+1}}e^{-\lambda_0 (n-k)r}\left( \|y\|_{\beta,\Delta_{k}}+ \|y\|_{\beta,\Delta_{k+1}}\right).
	\end{eqnarray}
\end{lemma}
\begin{proof}
First, for any $t\geq r$, by virtue of \eqref{yvar}, \eqref{estphi1} and the assumption $\bf H_2$ on $f$, the following inequalities hold 
	\begin{eqnarray*}
		\|y(t)\| &\leq& \|\Phi (t-r)y(r)\| + \int_r^t \|\Phi (t-s)f(y_s)\| ds + \Big\|\int_r^t\Phi (t-s)g(y_s) d  x (s) \Big\| \\
		&\leq& C_A e^{-\lambda  (t-r)} \|y(r)\| + \int_r^t C_A e^{-\lambda  (t-s)} \Big(C_f \|y_s\| + \|f(0)\|\Big) ds + \Big\|\int_r^t\Phi (t-s)g(y_s) d  x (s) \Big\|\\
		&\leq& C_A e^{-\lambda  (t-r)} \|y(r)\| + \frac{C_A}{\lambda } \|f(0)\| (1- e^{-\lambda (t-r)}) +C_AC_f \int_{r}^t e^{-\lambda (t-s)}  \|y\|_{\infty,[s-r,s]}ds  + \beta(t),
	\end{eqnarray*}
	where
\[\beta(t) := \begin{cases}
	\Big\|\int_r^t\Phi (t-s)g(y_s) d  x (s) \Big\|,&\quad t\geq r\\
	0,&\quad 0\leq t< r.
	\end{cases}.
\]	
	 Now assign  $\beta^*(t) := \sup_{[t-r,t]}\|\beta(s)\|$ for $t\geq r$.
It is easy to see that 
\begin{eqnarray*}
		\|y\|_{\infty,[t-r,t]} &=&\sup_{s\in[t-r,t]}\|y(s)\|\\
		&&\hspace{-2cm}\leq  C_A e^{-\lambda  (t-2r)}  \|y\|_{\infty,[0,r]}  + \frac{C_A}{\lambda } \|f(0)\| (1- e^{-\lambda (t-r)})  +C_AC_fe^{\lambda r} \int_{r}^t e^{-\lambda (t-s)}  \|y\|_{\infty,[s-r,s]}ds + \beta^*(t).
	\end{eqnarray*}
 Consequently, 
\begin{eqnarray*}
		e^{\lambda (t-r)}\|y\|_{\infty,[t-r,t]} &\leq&  C_A e^{\lambda r}  \|y\|_{\infty,[0,r]} + \frac{C_A}{\lambda} \|f(0)\| ( e^{\lambda (t-r)}-1) +e^{\lambda (t-r)} \beta^*(t) \notag\\
		&&\hspace*{2cm}+\; C_AC_fe^{\lambda r} \int_{r}^t e^{\lambda (s-r)}  \|y\|_{\infty,[s-r,s]}ds.
	\end{eqnarray*}	
Recall from \eqref{notationL} that $L:= C_AC_fe^{\lambda r}$. By applying the continuous Gronwall lemma~\ref{contgronwall} for the function $e^{\lambda (t-r)}\|y\|_{\infty,[t-r,t]}$ and performing several direct computations, we obtain
\begin{eqnarray*}
		e^{(\lambda- L) (t-r)}\|y\|_{\infty,[t-r,t]} &\leq&  C_A e^{\lambda r}   \|y\|_{\infty,[0,r]} + \frac{C_A \|f(0)\| }{\lambda-L}( e^{(\lambda-L)  (t-r)}-1) +e^{(\lambda-L)  (t-r)} \beta^*(t)\notag\\
		&& +L \int_{r}^t e^{(\lambda-L) (s-r)}  \beta^*(s)ds.
	\end{eqnarray*}	
Replacing $\lambda_0= \lambda-L$ as in \eqref{notationL} yields
\begin{eqnarray}\label{nt2003}
		&&e^{\lambda_0 (t-r)}\|y\|_{\infty,[t-r,t]}\notag\\
		 &&\hspace{2cm}\leq  C_A e^{\lambda r}   \|y\|_{\infty,[0,r]}+ \frac{C_A \|f(0)\| }{\lambda_0 }\left[ e^{\lambda_0  (t-r)}-1\right] +e^{\lambda_0  (t-r)} \beta^*(t) +L \int_{r}^t e^{\lambda_0 (s-r)}  \beta^*(s)ds.\notag\\
	\end{eqnarray}	
	
	Next, we are going to to estimate the quantities $\beta$ and $\beta^*$. First, take any $s\in \R$ such that $s\geq r$ and $\frac{s}{r}$ is not an integer. Put $n:= \lfloor s/r\rfloor$, the integer part of $\frac{s}{r}$.
Due to the definition of $\beta(s)$, the inequality \eqref{yloeve}, the estimates \eqref{estphi1}, \eqref{estphi2} and Proposition \ref{esth1}, we obtain
\allowdisplaybreaks
\begin{eqnarray*}
&&\hspace*{-1cm}\beta(s) \leq \\[5pt]
&\leq&\sum_{k=1}^{n-1}\left\| \int_{kr}^{(k+1)r}\Phi (s-u)g(y_u) d x(u) \right\|+\left\|\int_{nr}^s \Phi (s-u)g(y_u) d  x (u) \right\|\\
&\leq&
 \sum_{k=1}^{n-1}  r^\nu\ltn  x \rtn_{\nu,\Delta_k}\left[\| \Phi (s-kr)g(y_{kr}) \|+Kr^\beta \ltn \Phi(s-\cdot)g(y_{\cdot})\rtn_{\beta,\Delta_k} \right]\\
&&\hspace*{2cm}+\;  r^\nu\ltn  x \rtn_{\nu,[nr,s]}\left[\| \Phi (s-nr)g(y_{nr}) \|+Kr^\beta \ltn \Phi(s-\cdot)g(y_{\cdot})\rtn_{\beta,[nr,s]} \right]\\
&\leq&
\sum_{k=1}^{n-1}  r^\nu\ltn  x \rtn_{\nu,\Delta_k}\left[C_Ae^{-\lambda (s-kr)} (C_g\|y_{kr}\|+\|g(0)\|)+\right.\\
&&\hspace{2cm}+\left. Kr^\beta \left(\ltn \Phi(s-\cdot)\rtn_{\beta,\Delta_k} \|g(y_{\cdot})\|_{\infty,\Delta_k} + \|\Phi(s-\cdot)\|_{\infty,\Delta_k} \ltn g(y_{\cdot})\rtn_{\beta,\Delta_k}\right)\right]
\\
&&+ r^\nu\ltn  x \rtn_{\nu,\Delta_n}\left[C_Ae^{-\lambda (s-nr)} (C_g\|y_{nr}\|+\|g(0)\|)+\right.\\
&&\hspace{2cm}+\left. Kr^\beta \left(\ltn \Phi(s-\cdot)\rtn_{\beta,[nr,s]} \|g(y_{\cdot})\|_{\infty,[nr,s]} + \|\Phi(s-\cdot)\|_{\infty,[nr,s]} \ltn g(y_{\cdot})\rtn_{\beta,[nr,s]}\right)\right]
\\
&\leq&
\sum_{k=1}^{n}  r^\nu\ltn  x \rtn_{\nu,\Delta_k}\left[C_Ae^{-\lambda (s-kr)} (C_g\sup_{-r\leq u\leq0}\|y(kr+u)\|+\|g(0)\|)+\right.\\
&&\hspace{2cm}+\left. Kr^\beta \left(\|A\|C_Ae^{-\lambda(s-kr-r)}r^{1-\beta} \|g(y_{\cdot})\|_{\infty,\Delta_k} + C_Ae^{-\lambda(s-kr-r)} \ltn g(y_{\cdot})\rtn_{\beta,\Delta_k}\right)\right]
\\
&\leq&
\sum_{k=1}^{n}  r^\nu\ltn  x \rtn_{\nu,\Delta_k}\left[C_Ae^{-\lambda (s-kr)} (C_g\|y\|_{\infty,\Delta'_k}+\|g(0)\|)+\right.\\
&&\hspace{1cm}+\left. Kr^\beta \left(\|A\|C_Ae^{-\lambda(s-kr-r)}r^{1-\beta} (C_g\|y\|_{\infty,\Delta'_k} +\|g(0)\|) + C_Ae^{-\lambda(s-kr-r)}C_g \ltn y\rtn_{\beta,\Delta'_k}\right)\right]\\
&\leq&
\sum_{k=1}^{n}  r^\nu\ltn  x \rtn_{\nu,\Delta_k}KC_A (1+\|A\|r)e^{-\lambda (s-kr-r)}\left[C_g\left(\|y\|_{\infty,\Delta'_k}+r^\beta\ltn y\rtn_{\beta,\Delta'_k}\right) +\|g(0)\|\right]\\
&\leq&
\sum_{k=1}^{n}  r^\nu\ltn  x \rtn_{\nu,\Delta_k}KC_A(1+\|A\|r)e^{-\lambda (s-kr-r)} \left[C_g\left( \|y\|_{\beta,\Delta_{k-1}}+ \|y\|_{\beta,\Delta_k}\right) + \|g(0)\|\right].
\end{eqnarray*}
Taking the supremum on $[s-r,s]$ yields
\begin{eqnarray}\label{int210}
	\beta^*(s) 
	&\leq
	\sum_{k=1}^{n} e^{2\lambda r}  r^\nu\ltn  x \rtn_{\nu,\Delta_k}KC_A(1+\|A\|r)e^{-\lambda (s-kr)} \left[C_g\left( \|y\|_{\beta,\Delta_{k-1}}+ \|y\|_{\beta,\Delta_k}\right) + \|g(0)\|\right].\notag\\
	\end{eqnarray}
	Therefore, given that $\lambda_0 = \lambda-L > 0$,
	\begin{eqnarray}\label{int211}
\beta^*(s) e^{\lambda_0 (s-r)}
	&\leq&
	 \sum_{k=1}^{\lfloor s/r\rfloor} e^{2\lambda r}r^\nu KC_A(1+\|A\|r) \ltn  x \rtn_{\nu,\Delta_k}e^{-Ls+\lambda kr} \left[C_g\left( \|y\|_{\beta,\Delta_{k-1}}+ \|y\|_{\beta,\Delta_k}\right) + \|g(0)\|\right].\notag\\
	\end{eqnarray}
Combining \eqref{int211} with \eqref{nt2003} and applying the same arguments as in \cite{duchong19}, we obtain for any fixed $n$ and any $t\in[nr,(n+1)r)$, 
\begin{eqnarray*}
&&e^{\lambda_0 (t-r)}\|y\|_{\infty,[t-r,t]}\\
&\leq&  C_A e^{\lambda r} \|y\|_{\infty,[0,r]}+ \frac{C_A \|f(0)\| }{\lambda_0}( e^{\lambda_0 (t-r)}-1) \notag\\
 && +\sum_{k=1}^{n} KC_Ae^{2\lambda r} r^\nu (1+\|A\|r)\ltn  x \rtn_{\nu,\Delta_k}e^{\lambda_0 kr}e^{-L(t-kr)} \left[C_g\left( \|y\|_{\beta,\Delta_{k-1}}+ \|y\|_{\beta,\Delta_k}\right) + \|g(0)\|\right] \\
&&\hspace{-1cm}+L \int_{r}^{t}  \sum_{k=1}^{\lfloor s/r\rfloor}KC_Ae^{2\lambda r} r^\nu (1+\|A\|r) \ltn  x \rtn_{\nu,\Delta_k}e^{\lambda_0 kr}e^{-L(s-kr)} \left[C_g\left( \|y\|_{\beta,\Delta_{k-1}}+ \|y\|_{\beta,\Delta_k}\right) + \|g(0)\|\right]ds\\
 &\leq& C_A e^{\lambda r}   \|y\|_{\infty,[0,r]}+ \frac{C_A \|f(0)\| }{\lambda_0}( e^{\lambda_0  t}-1) \notag\\
&& +\sum_{k=1}^{n}KC_Ae^{2\lambda r} r^\nu (1+\|A\|r)\ltn  x \rtn_{\nu,\Delta_k}e^{\lambda_0 kr} \left[C_g\left( \|y\|_{\beta,\Delta_{k-1}}+ \|y\|_{\beta,\Delta_k}\right) + \|g(0)\|\right] \\
&&\times \left(e^{-L(t-kr)} +L \int_{kr}^{t}e^{-L(s-kr)}ds \right)\\
 &\leq& C_A e^{\lambda r}  \|y\|_{\infty,[0,r]}+ \frac{C_A \|f(0)\| }{\lambda_0 }( e^{\lambda_0  t}-1) \notag\\
&& +\sum_{k=1}^{n} KC_Ae^{2\lambda r} r^\nu (1+\|A\|r)\ltn  x \rtn_{\nu,\Delta_k}e^{\lambda_0 kr} \left[C_g\left( \|y\|_{\beta,\Delta_{k-1}}+ \|y\|_{\beta,\Delta_k}\right) + \|g(0)\|\right]. 
	\end{eqnarray*}
Consequently, we have
\begin{eqnarray}
\|y\|_{\infty,[t-r,t]}e^{\lambda_0 (t-r)}
 &\leq&  C_A e^{\lambda r}  \|y\|_{\infty,[0,r]}+ \frac{C_A \|f(0)\| }{\lambda_0 }\left[ e^{\lambda_0  (t-r)}-1\right]  \notag\\
 &&+\; KC_Ae^{2\lambda r} r^\nu (1+\|A\|r)\|g(0)\|\sum_{k=1}^{n} \ltn x\rtn_{\nu,\Delta_k}e^{\lambda_0 kr}\notag\\
&& +\; C_gKC_Ae^{2\lambda r} r^\nu (1+\|A\|r)\sum_{k=1}^{n} \ltn x\rtn_{\nu,\Delta_k}e^{\lambda_0 kr}\left( \|y\|_{\beta,\Delta_{k-1}}+ \|y\|_{\beta,\Delta_k}\right). \notag
\end{eqnarray}
This shows that there exists a positive number $M_2$ such that 
\begin{eqnarray}
\|y\|_{\infty,[t-r,t]}&\leq& M_2 e^{-\lambda_0 nr}  \|y\|_{\infty,[0,r]}+ (\|f(0)\|\vee \|g(0)\|) M_2\sum_{k=0}^{n-1} (1+ \ltn x\rtn_{\nu,\Delta_{k+1}})e^{-\lambda_0 (n-k)r} \notag\\
&&+\; C_gM_1 \sum_{k=0}^{n-1} \ltn x\rtn_{\nu,\Delta_{k+1}}e^{-\lambda_0 (n-k)r}\left( \|y\|_{\beta,\Delta_{k}}+ \|y\|_{\beta,\Delta_{k+1}}\right)\notag
	\end{eqnarray}
	for all $t\in[nr,(n+1)r)$, $n\in\N$. Due to continuity, this inequality also holds for $t= (n+1)r$ and for all $t\in \Delta_n$. Thus \eqref{yinfty} is proved and so is \eqref{yinftyn} by assigning $t:= (n+1)r$ in \eqref{yinfty}.
\end{proof}

\begin{remark}
Inequalities \eqref{yinftyn} and \eqref{ykk} show that the supremum norm of the solution on $\Delta_n$ depends not only on itself (up to a coefficient dependent on $x$) but also on the H\"older norm of the solution on the previous intervals. This is different from the non-delay case (see \cite{duchong19}) and is very challenging to deal with. We therefore need to estimate the $\beta-$H\"older norm of $y$ in the similar form to \eqref{yinftyn} in the following Lemma. \end{remark}
Assign
\begin{equation}\label{notationM3}
M_3:=Kr^\nu e^{(L_f+4\lambda) r}\Big[1+C_AL_fr(1+\|A\|r)\Big].
\end{equation}
 
\begin{lemma}\label{est.ybeta} For any $n\in \N$, $n\geq 1$,
there exists a positive constant $M_4$ independent of $n$, such that
the $\beta-$H\"older norm of the solution of \eqref{YDDE} on $\Delta_n$, can be estimated as follows
\begin{eqnarray}\label{ybetan}
r^\beta\ltn y\rtn_{\beta,\Delta_n} 
&\leq& M_4e^{-\lambda_0 nr}  \|y\|_{\infty,[0,r]}+M_4(\|f(0)\|\vee \|g(0)\|)  \sum_{k=0}^{n-1} (1+ \ltn x\rtn_{\nu,\Delta_{k+1}})e^{-\lambda_0 (n-k)r}\notag\\
&&+\;C_gM_3  \sum_{k=0}^{n-1} \ltn x\rtn_{\nu,\Delta_{k+1}}e^{-\lambda_0 (n-k)r}\left( \|y\|_{\beta,\Delta_{k}}+ \|y\|_{\beta,\Delta_{k+1}}\right),
\end{eqnarray}
where the constant $M_3$ is defined by the formula \eqref{notationM3}.
\end{lemma}
\begin{proof}
We fix $v\in \Delta_n$, and consider  $s,t\in[nr,v]$, $s<t$. Observe that
\begin{eqnarray}
\|y(t)-y(s)\| &=&\left\| \int_s^t [Ay(u)+f(y_u)]du + \int_s^tg(y_u)dx(u)\right\|\notag\\
&\leq& \|f(0)\|(t-s)+(\|A\|+C_f) \int_s^t\|y_u\|du+ \left\|\int_s^tg(y_u)dx(u)\right\|\notag.
\end{eqnarray}
 Since $u\in [s,t]\subset [nr,v]$, it follows that $\|y_u\|\leq \|y\|_{\infty,[s-r,s]} + (u-s)^\beta\ltn y\rtn_{\beta,[s,u]}$. In addition, by the Young-Loeve inequality  \eqref{yloeve}, Proposition~\ref{esth1} and the definition 
of the norm $\|\cdot\|_{\beta,[a,b]}$, 
\begin{eqnarray*}
\left\|\int_s^tg(y_u)dx(u)\right\|&\leq&(t-s)^\nu\ltn x\rtn_{\nu,\Delta_n}K\left[C_g\left(\|y\|_{\beta,\Delta_{n-1}}+\|y\|_{\beta,\Delta_{n}}\right) +\|g(0)\|\right].
\end{eqnarray*}
Therefore, due to $L_f =\|A\|+C_f$, it follows that
\begin{eqnarray}\label{yadd2}
\frac{\|y(t)-y(s)\|}{(t-s)^\beta} &\leq &\|f(0)\|r^{1-\beta}+L_f r^{1-\beta}\|y\|_{\infty,[s-r,s]}+ L_f\int_s^t \frac{(u-s)^\beta}{(t-s)^\beta}\ltn y\rtn_{\beta,[s,u]}du \notag\\
&&+\; Kr^{\nu-\beta}\ltn x\rtn_{\nu,\Delta_n}\left[C_g\left(\|y\|_{\beta,\Delta_{n-1}}+\|y\|_{\beta,\Delta_{n}}\right) +\|g(0)\|\right]\notag\\
&\leq &\|f(0)\|r^{1-\beta}+L_f r^{1-\beta}\|y\|_{\infty,[s-r,s]}+L_f \int_s^t \ltn y\rtn_{\beta,[s,u]}du \notag\\
&&+\; Kr^{\nu-\beta}\ltn x\rtn_{\nu,\Delta_n}\left[C_g\left(\|y\|_{\beta,\Delta_{n-1}}+\|y\|_{\beta,\Delta_{n}}\right) +\|g(0)\|\right]\notag\\
&\leq &(\|f(0)\|\vee \|g(0)\|) \max\{ r^{1-\beta},  Kr^{\nu-\beta}\}(1+\ltn x\rtn_{\nu,\Delta_n} ) + L_f r^{1-\beta}\|y\|_{\infty,[s-r,s]}\notag\\
&&+\; C_gKr^{\nu-\beta}\ltn x\rtn_{\nu,\Delta_n}\left(\|y\|_{\beta,\Delta_{n-1}}+\|y\|_{\beta,\Delta_{n}}\right)+ L_f\int_s^t \ltn y\rtn_{\beta,[s,u]}du .
\end{eqnarray}
Together with \eqref{yinfty}, and with the notation $M'_2:=L_frM_2+ (r \vee K r^\nu)e^{\lambda_0 r}>0$, the following estimate holds for all  $[s,t]\subset [nr,v]$
\allowdisplaybreaks
\begin{eqnarray}
&&r^\beta\frac{\|y(t)-y(s)\|}{(t-s)^\beta} \notag\\
&\leq &(\|f(0)\|\vee \|g(0)\|) \max\{ r,  Kr^{\nu}\}(1+\ltn x\rtn_{\nu,\Delta_n} ) \notag\\
&&+\; L_fr M_2e^{-\lambda_0 nr}   \|y\|_{\infty,[0,r]}+ (\|f(0)\|\vee \|g(0)\|) L_frM_2\sum_{k=0}^{n-1} (1+ \ltn x\rtn_{\nu,\Delta_{k+1}})e^{-\lambda_0 (n-k)r} \notag\\
&&+\; C_gL_frM_1 \sum_{k=0}^{n-1} \ltn x\rtn_{\nu,\Delta_{k+1}}e^{-\lambda_0 (n-k)r}\left( \|y\|_{\beta,\Delta_{k}}+ \|y\|_{\beta,\Delta_{k+1}}\right)\notag\\
&&+\; C_g Kr^{\nu}\ltn x\rtn_{\nu,\Delta_n}\left(\|y\|_{\beta,\Delta_{n-1}}+\|y\|_{\beta,\Delta_{n}}\right)+ L_fr^\beta\int_s^t \ltn y\rtn_{\beta,[s,u]}du \notag\\
&\leq& M'_2e^{-\lambda_0 nr}  \|y\|_{\infty,[0,r]} +(\|f(0)\|\vee \|g(0)\|)  M'_2\sum_{k=0}^{n-1} (1+ \ltn x\rtn_{\nu,\Delta_{k+1}})e^{-\lambda_0 (n-k)r}\notag\\
&&+\; \Big(C_g L_f r M_1 + C_g Kr^\nu e^{\lambda_0 r}\Big)\sum_{k=0}^{n-1} \ltn x\rtn_{\nu,\Delta_{k+1}}e^{-\lambda_0 (n-k)r}\left( \|y\|_{\beta,\Delta_{k}}+ \|y\|_{\beta,\Delta_{k+1}}\right)\notag\\
&&+\; L_f\int_s^t r^\beta\ltn y\rtn_{\beta,[s,u]}du. \notag 
\end{eqnarray}
This implies
\begin{eqnarray}
r^\beta\ltn y\rtn_{\beta,[nr,v]} 
&\leq& M'_2e^{-\lambda_0 nr}  \|y\|_{\infty,[0,r]} +M'_2(\|f(0)\|\vee \|g(0)\|)  \sum_{k=0}^{n-1} (1+ \ltn x\rtn_{\nu,\Delta_{k+1}})e^{-\lambda_0 (n-k)r}\notag\\
&&+\; \Big(C_g L_f r M_1 + C_g Kr^\nu e^{\lambda_0 r}\Big)  \sum_{k=0}^{n-1} \ltn x\rtn_{\nu,\Delta_{k+1}}e^{-\lambda_0 (n-k)r}\left( \|y\|_{\beta,\Delta_{k}}+ \|y\|_{\beta,\Delta_{k+1}}\right)\notag\\
&&+\; L_f\int_{nr}^vr^\beta \ltn y\rtn_{\beta,[nr,u]}du. \notag
\end{eqnarray}
Applying the Gronwall Lemma~\ref{contgronwall} to the function $r^\beta\ltn y\rtn_{\beta,[nr,\cdot]}$ yields
\begin{eqnarray*}
r^\beta\ltn y\rtn_{\beta,[nr,v]} 
&\leq& \Big[M'_2e^{-\lambda_0 nr}   \|y\|_{\infty,[0,r]} +M'_2(\|f(0)\|\vee \|g(0)\|)  \sum_{k=0}^{n-1} (1+ \ltn x\rtn_{\nu,\Delta_{k+1}})e^{-\lambda_0 (n-k)r}+\Big.\notag\\
&&\hspace{-1cm}+\; \Big. \Big(C_g L_f r M_1 + C_g Kr^\nu e^{\lambda_0 r}\Big) \sum_{k=0}^{n-1} \ltn x\rtn_{\nu,\Delta_{k+1}}e^{-\lambda_0 (n-k)r}\left( \|y\|_{\beta,\Delta_{k}}+ \|y\|_{\beta,\Delta_{k+1}}\right)\Big]\times\notag\\
&&\times \Big(1+L_f\int_{nr}^ve^{L_f(v-u)}du\Big).\notag\\
\end{eqnarray*}
Consequently, 
\begin{eqnarray*}
r^\beta\ltn y\rtn_{\beta,\Delta_n} 
&\leq& \Big[M'_2e^{-\lambda_0 nr}  \|y\|_{\infty,[0,r]} +M'_2(\|f(0)\|\vee \|g(0)\|)  \sum_{k=0}^{n-1} (1+ \ltn x\rtn_{\nu,\Delta_{k+1}})e^{-\lambda_0 (n-k)r}+\Big.\notag\\
&&\hspace{-1cm}+\; \Big. C_gKr^\nu\Big( e^{\lambda_0 r}+L_fr^{1-\nu}M_1 \Big)  \sum_{k=0}^{n-1} \ltn x\rtn_{\nu,\Delta_{k+1}}e^{-\lambda_0 (n-k)r}\left( \|y\|_{\beta,\Delta_{k}}+ \|y\|_{\beta,\Delta_{k+1}}\right)\Big] e^{L_fr}.
\end{eqnarray*}
Assigning $M_4 := e^{L_fr}M'_2$ and taking into account \eqref{notationM1}, \eqref{notationM3}  we finally obtain \eqref{ybetan}. 
\end{proof}
\medskip
 
 Now we are in a position to prove the main result of the paper on the existence of a random pullback attractor for the stochastic Young differential delay equation \eqref{SDDE}, which is formulated as follows.
\begin{theorem}\label{attractor}
Consider the system \eqref{SDDE}
\[
dy(t) = [Ay(t)+ f(y_t)]dt + g(y_t)dZ(t),\quad y_0=\eta\in C^{0,\beta_0}([-r,0],\R^d).
\]
Assume that the conditions  $\bf {H_1, H_2,H_3}$ hold, and additionally
\[
C_AC_f<\lambda e^{-\lambda r}.
\] 
Then there exists $\eps>0 $ such that for $C_g<\eps$, the generated random dynamical system of \eqref{SDDE} possesses a random pullback attractor $\mathcal{A}(x)$ which is in $\cC^{\beta}([-r,0],\R^d) \subset \cC^{0,\beta_0}([-r,0],\R^d)$.
\end{theorem}
\begin{proof}
As noticed before, the equation \eqref{SDDE} is understood in the path-wise sense with Riemann-Stieltjes integration and Young integration. We consider the deterministic equation \eqref{YDDE}
\[
dy(t) = [Ay(t)+ f(y_t)]dt + g(y_t)dx(t),\quad y_0=\eta\in \cC^{0,\beta_0}([-r,0],\R^d),
\]
which is a representative path-wise equation of  the stochastic equation \eqref{SDDE}. With a little abuse of notation, we will denote by $y(\cdot)$ the solution of system \eqref{SDDE} and also of system \eqref{YDDE}.

Notice that the condition $C_AC_f<\lambda e^{-\lambda r}$ is equivalent to the condition $\lambda_0=\lambda-L>0$.
Put  $M_5:=M_1+M_3$ and $M_6:= M_2+M_4$. 
Due to \eqref{yinftyn} and \eqref{ybetan}, we obtain for any $n\geq 1$ 
\begin{eqnarray}\label{yinfbetan}
\| y\|_{\beta,\Delta_n} 
&\leq& M_6e^{-\lambda_0 nr}   \|y\|_{\infty,[0,r]} +M_6(\|f(0)\|\vee \|g(0)\|)  \sum_{k=0}^{n-1} (1+ \ltn x\rtn_{\nu,\Delta_{k+1}})e^{-\lambda_0 (n-k)r}\notag\\
&&+\; C_gM_5  \sum_{k=0}^{n-1} \ltn x\rtn_{\nu,\Delta_{k+1}}e^{-\lambda_0 (n-k)r}\left( \|y\|_{\beta,\Delta_{k}}+ \|y\|_{\beta,\Delta_{k+1}}\right).
\end{eqnarray}
Now, we apply Proposition~\ref{esty1}. Assign for $[a,b]\subset \R$, 
\[
 F(x,[a,b]):= 
1+[2(K+1)C_g(b-a)^\nu]^{\frac{1}{\nu-\beta} } \ltn  x\rtn^{\frac{1}{\nu-\beta}}_{\nu,[a,b]}.
\]
The estimate \eqref{ykk} of Proposition~\ref{esty1} then has the form
\begin{eqnarray*}
	\|y\|_{\beta,\Delta_k} \leq e^{4C_fr+\kappa N_{k}(x)} \left[\|y\|_{\beta,\Delta_{k-1}} +\left(4r\|f(0)\|+\frac{\|g(0)\|}{C_g}\right)\right]  - \left(4r\|f(0)\|+\frac{\|g(0)\|}{C_g}\right),
\end{eqnarray*}
where $\kappa=4L_fr+2$, and $N_k(x)$ is the counting function of stopping times of the greedy sequence on $\Delta_k$ described in the proof of Proposition~\ref{esty1}. By \eqref{Nn}, $N_k(x)  \leq F(  x,\Delta_k)$. Hence
\begin{eqnarray*}
	\|y\|_{\beta,\Delta_k} + \|y\|_{\beta,\Delta_{k+1}} \leq  \Big( 1+e^{4C_fr+\kappa N_{k+1}(x)}\Big) \|y\|_{\beta,\Delta_{k}}+ e^{4C_fr+\kappa N_{k+1}(x)}\left(4r\|f(0)\|+\frac{\|g(0)\|}{C_g}\right).
\end{eqnarray*}
Together with \eqref{yinfbetan}, this yields
\allowdisplaybreaks
\begin{eqnarray}
&&e^{\lambda_0 nr}\| y\|_{\beta,\Delta_n}\notag \\
&\leq& M_6  \|y\|_{\infty,[0,r]}+M_6(\|f(0)\|\vee \|g(0)\|)  \sum_{k=0}^{n-1} (1+ \ltn x\rtn_{\nu,\Delta_{k+1}})e^{\lambda_0 kr}\notag\\
&&+\;C_gM_5  \sum_{k=0}^{n-1} \ltn x\rtn_{\nu,\Delta_{k+1}}e^{\lambda_0 kr}\Big[ \Big( 1+e^{4C_fr+\kappa N_{k+1}(x)}\Big) \|y\|_{\beta,\Delta_{k}}+ e^{4C_fr+\kappa N_{k+1}(x)}\Big(4r\|f(0)\|+\frac{\|g(0)\|}{C_g}\Big)\Big] \notag\\
&&\leq \;  M_6  \|y\|_{\infty,[0,r]}+ M_8(\|f(0)\|\vee \|g(0)\|)  \sum_{k=0}^{n-1} e^{\lambda_0 kr}(1+ \ltn x\rtn_{\nu,\Delta_{k+1}})e^{\kappa F(x,\Delta_{k+1})}\notag\\
&&\hspace{2cm}+\; C_gM_7  \sum_{k=0}^{n-1} \ltn x\rtn_{\nu,\Delta_{k+1}}\left(1+ e^{\kappa F(x,\Delta_{k+1})}\right)e^{\lambda_0 kr}\|y\|_{\beta,\Delta_{k}}\notag\\
&&\leq \;  M_8   \|y\|_{\infty,[0,r]}+ M_8(\|f(0)\|\vee \|g(0)\|)  \sum_{k=0}^{n-1} e^{\lambda_0 kr}H(x,\Delta_{k+1})\notag\\
&&\hspace{2cm}+\; C_gM_7  \sum_{k=0}^{n-1} G(x,\Delta_{k+1})e^{\lambda_0 kr}\|y\|_{\beta,\Delta_{k}},
\label{yinfbetan1}
\end{eqnarray}
where we used the notations
\begin{eqnarray}
M_7&:=&M_5e^{4C_f r},\\
M_8&:=& 1+M_6 + M_5e^{4C_f r}(4C_gr+1),\\
 G(x,[a,b])&:=& \ltn x\rtn_{\nu,[a,b]} \left(1+ e^{\kappa F(x,[a,b])}\right),\\
  H(x,[a,b]) &:=&(1+ \ltn x\rtn_{\nu,[a,b]})e^{\kappa F(x,[a,b])}.
\end{eqnarray}
Due to the discrete Gronwall Lemma~\ref{gronwall}, we derive from \eqref{yinfbetan1} that
\begin{eqnarray}\label{yinfbetan3}
e^{\lambda_0 nr}\| y\|_{\beta,\Delta_n} 
&\leq& M_8\ \|y\|_{\beta,\Delta_0} \prod_{k=0}^{n-1}\left[1+C_gM_7 G(x,\Delta_{k+1})\right] \notag\\
&&\hspace{-1cm}+\; M_8(\|f(0)\|\vee \|g(0)\|)   \sum_{k=0}^{n-1}e^{\lambda_0 kr} H(x,\Delta_{k+1}) \prod_{j=k+1}^{n-1}\left[1+C_gM_7 G(x,\Delta_{j+1})\right].
\end{eqnarray}
From the construction of the random dynamical system generated by \eqref{SDDE} in Section~\ref{sec.genRDS}, it follows that $G(x,\Delta_k) = G(\theta_{kr}x, [0,r])$ and $H(x,\Delta_k) = H(\theta_{kr}x, [0,r])$. Hence, by writing the solution of \eqref{SDDE} in full form $y(\cdot,x,\eta)$ indicating the dependence on the driving path $x$ and the initial condition $\eta$, we obtain
 \begin{eqnarray}
e^{\lambda_0 nr}\| y(\cdot,x,\eta)\|_{\beta,\Delta_n} 
&\leq&  M_8\ \|y\|_{\beta,\Delta_0}\prod_{k=0}^{n-1}\left[1+C_gM_7 G(\theta_{(k+1)r}x,[0,r])\right] \notag\\
&&\hspace{-5cm}+\; M_8(\|f(0)\|\vee \|g(0)\|)   \sum_{k=0}^{n-1} e^{\lambda_0 kr} H(\theta_{(k+1)r}x,[0,r]) \prod_{j=k+1}^{n-1}\left[1+C_gM_7 G(\theta_{(j+1)r}x,[0,r])\right].\label{eq.pullback}
\end{eqnarray}
By \eqref{ybeta0r}, there exists a positive constant $M_{9}$ independent of $n$ such that
\[  
M_8\ \|y\|_{\beta,\Delta_0}
  \leq 
 M_{9} \Big( 1+ \ltn x \rtn_{\nu,[0,r]}\Big) \Big(1+\|\eta\|_{\beta_0,[-r,0]} \Big) e^{\kappa N_0(x)}.
\] 
Put $F_0(x,[0,r]):=1+[2(K_0+1)C_gr^\nu]^{\frac{1}{\nu-\beta_0} } \ltn  x\rtn^{\frac{1}{\nu-\beta_0}}_{\nu,[0,r]}$. By \eqref{Nk0} we have $N_0(x) \leq F_0(x,[0,r])$.
 Now, replacing $x$ by $\theta_{-(n+1)r}x$ in \eqref{eq.pullback} yields
 \begin{eqnarray}\label{y}
  \|y(\cdot,\theta_{-(n+1)r}x,\eta)\|_{\beta,\Delta_{n}}
 &\leq &M_{9} \Big(1+\|\eta\|_{\beta_0,[-r,0]} \Big) \Big( 1+ \ltn \theta_{-(n+1)r} x \rtn_{\nu,[0,r]}\Big) e^{\kappa F_0(\theta_{-(n+1)r}x,[0,r])}e^{-\lambda_0 nr}\times\notag\\
&&\hspace{2cm} \times   \prod_{k=1}^{n}\Big[1+  C_gM_7G(\theta_{-kr}x,[0,r])\Big] +\notag\\
  &&\hspace{-4cm}+\; M_8(\|f(0)\|\vee \|g(0)\|)  \sum_{k=1}^{n}  e^{-\lambda_0 kr} H(\theta_{-kr}x,[0,r])\prod_{i=1}^{k-1}\Big(1+C_gM_7G(\theta_{-ir} x,[0,r])\Big).
 \end{eqnarray} 
On the other hand, by using the inequality $\log (1+ae^b)\leq a+b$ for $a,b>0$ we can show that
\begin{eqnarray}\label{lebesguebound}
\log \Big(1+C_gM_7G(x,[0,r])\Big)&=&\log \Big(1+C_gM_7\ltn x\rtn_{\nu,[0,r]} \left(1+ e^{\kappa F(x,[0,r])}\right)\Big) \notag\\
						&\leq& 2C_gM_7\ltn x\rtn_{\nu,[0,r]} +\kappa F(x,[0,r]) \notag\\ 
						&\leq& \kappa+2C_gM_7\ltn x\rtn_{\nu,[0,r]} + \kappa [4(K+1)C_gr^\nu]^{\frac{1}{\nu-\beta} } \ltn  x\rtn^{\frac{1}{\nu-\beta}}_{\nu,[0,r]}.
\end{eqnarray}
Therefore $\log \Big[1+C_gM_7G(x,[0,r])\Big]$ is integrable due to the condition \eqref{Gamma}. Put 
\begin{equation}\label{meanvalueG}
\hat{G} = E\log \Big[1+C_gM_7G(x,[0,r])\Big].
\end{equation}
 Due to Birkhoff ergodic theorem, the following equality holds almost surely
\begin{eqnarray*}
		\limsup \limits_{n \to \infty}\frac{1}{n}\log \prod_{k=1}^{n} \Big[1+C_gM_7G(\theta_{-kr}x,[0,r])\Big] &=&\limsup \limits_{n \to \infty}\frac{1}{n}\sum_{k=1}^{n}\log \Big[1+C_gM_7G(\theta_{kr}x,[0,r])\Big]  \\
		&=& \hat{G}.
	\end{eqnarray*}
Similarly, one can show that $\log H$ is integrable. Furthermore, $F_0$ is integrable due to the fact that $\beta_0<\beta$, hence $\log {\tilde F}_0(x,[0,r])$ is integrable, where  ${\tilde F}_0(x,[0,r]):= \Big( 1+ \ltn x \rtn_{\nu,[0,r]}\Big) e^{\kappa F_0(x,[0,r])}$. Therefore, by the temperedness of  integrable random variables (see Arnold~\cite[Proposition 4.1.3, p.\ 165]{arnold}) the following equalities hold almost surely
\begin{eqnarray*}
		\limsup \limits_{n \to \infty}\frac{\log H(\theta_{nr} x,[0,r])}{n} = \limsup \limits_{n \to \infty}\frac{\log H(\theta_{-nr}x,[0,r])}{n} = 0 
\end{eqnarray*}
	and
\begin{eqnarray*}
		\limsup \limits_{n \to \infty}\frac{\log {\tilde F}_0(\theta_{-nr} x,[0,r])}{n} = 0. 
	\end{eqnarray*}
	Observe that $\log (1+C_gM_7G(x,[0,r]))$, as a function of $C_g$, converges pointwise to zero as $C_g$ tends to zero. Due to  \eqref{lebesguebound} and Lebesgue's dominated convergence theorem, the value $\hat{G}$ also converges to zero as $C_g$ tends to zero. Therefore there exists $\eps>0$ such that if $C_g < \eps$ then  $\hat{G} <\lambda_0 r$.  Fix $0<2\delta < \lambda_0r- \hat{G}$; there exists $n_0=n_0(\delta, x)$ such that for all $n\geq n_0$,
	\[
	e^{(-\delta+\hat{G})n}\leq  \prod_{k=1}^{n} \Big[1+C_gM_7G( \theta_{-kr}x,[0,r])\Big], \;\;  \prod_{k=1}^{n} \Big[1+C_gM_7G(\theta_{kr}x,[0,r])\Big]\leq e^{(\delta+\hat{G})n}
	\]
	and
	\[
	e^{-\delta n}\leq {\tilde F}_0(\theta_{-nr}x,[0,r]), H(\theta_{-nr}x,[0,r]),\;\; H( \theta_{nr}x,[0,r]) \leq e^{\delta n}.
	\]
	Consequently, from \eqref{y} it follows that for all $n\geq n_0$ 
	\begin{eqnarray}\label{y1}
  \|y(\cdot,\theta_{-(n+1)r}x,\eta)\|_{\beta,\Delta_{n}}
 &\leq &M_{9} \Big(1+\|\eta\|_{\beta_0,[-r,0]} \Big) e^{-\lambda_0 nr}  e^{(2\delta+\hat{G})n} \notag\\
  &&\hspace{-4cm}+\;M_8(\|f(0)\|\vee \|g(0)\|) \sum_{k=1}^{\infty}  e^{-\lambda_0 kr} H(\theta_{-kr}x,[0,r])\prod_{i=1}^{k}\Big(1+C_gM_7G(\theta_{-ir} x,[0,r])\Big).
 \end{eqnarray}

Now using the condition \eqref{Gamma} and following the arguments in \cite[Theorem \ 3.5]{duchong19}, we can prove that there exists a positive number $\eps$ and a positive tempered random variable $b(x)$ such that under condition $C_g<\eps$, there exits for any tempered compact random set $\hat {D}(\cdot)\in \cD$ a time $t(x,\hat{D})>0$ such that for all  $t\geq t(x,\hat{D})$ and all $\eta\in \hat {D}(\theta_{-t}x)$ the solution satisfies
\[
\|y(\cdot,\theta_{-t}x,\eta)\|_{\beta,[t-r,t]}\leq b(x).
\]
In fact, one may choose 
\[
b(x) := 1 + M_8(\|f(0)\|\vee \|g(0)\|) \sum_{k=1}^{\infty}  e^{-\lambda_0 kr} H(\theta_{-kr}x,[0,r])\prod_{i=1}^{k}\Big(1+C_gM_7G(\theta_{-ir} x,[0,r])\Big),
\]
and the temperedness of $b(\cdot)$ is proved in \cite{duchong19}.
For convenience of the reader we give an improved proof of temperedness of  $b(\cdot)$ which is based on Lemmas \ref{lem.tempered1} and \ref{lem.tempered2} and is shorter than that of \cite{duchong19}. Namely, since $0< 2\delta < \lambda_0r-\hat{G}$, it follows that
\[
b(x) - 1 \leq
\Big[M_8(\|f(0)\|\vee \|g(0)\|) \sum_{k=1}^{\infty}  e^{-\delta k}H(\theta_{-kr}x,[0,r])\Big] 
\Big[ \sum_{k=1}^{\infty}  e^{-(\hat{G} +\delta) k}\prod_{i=1}^{k}\Big(1+C_gM_7G(\theta_{-ir} x,[0,r])\Big)\Big].
\]
The first multiplier in the right-hand side is tempered due to Lemma~\ref{lem.tempered2}$(ii)$; the second multiplier is tempered due to Lemma~\ref{lem.tempered2}$(i)$. Consequently, the function in the right-hand side is tempered due to Lemma~\ref{lem.tempered1}$(i)$. This implies that $b$ is tempered because of 
Lemma~\ref{lem.tempered1}$(ii)$.

 Note that, for each $x\in \cC^{0,\nu}_0(\R,\R^m)$ in the canonical representation space of $Z$,  the closed ball $\mathcal{B} (x) =\{\eta\in \cC^{0,\beta_0}([-r,0],\R^d)\mid\ \|\eta\|_{\beta,[-r,0]}\leq b(x)\}$ is compact in $\cC^{0,\beta_0}([-r,0],\R^d)$. Thus we proved that there exists a compact absorbing  random set $\mathcal{B} (x)$ with respect to the universe of tempered compact random sets. Moreover,  $\mathcal{B} (x)$ is a subset of $\cC^{\beta}([-r,0],\R^d)$.  Therefore, $\varphi$ possesses a random pullback attractor $\mathcal{A} (x) \subset \mathcal{B} (x)$ (see  \cite[Theorem 2.4]{GAKLBSch2010},  \cite[Theorem\ 3.5]{flandolibjorn}). Clearly, $\mathcal{A} (x) \subset \cC^{\beta}([-r,0],\R^d) \subset \cC^{0,\beta_0}([-r,0],\R^d)$.
 \end{proof}

The inequality \eqref{y1} provides us with a tool to make further conclusions on the dynamics of the random system generated by \eqref{SDDE} in case we have some additional information on the coefficient functions $f$ and $g$ as the following corollary shows. Recall from \cite{crauelscheutzow} that a {\em random forward attractor} is defined in a similar manner as the random pullback attractor given in Definition~\ref{dfn.attractor}, namely we replace the pullback attraction condition $(iii)$ of Definition~\ref{dfn.attractor} by the forward attraction one, i.e.\ for every $\hat{D} \in \cD$,
$\lim \limits_{t \to \infty} d(\varphi(t,x)
\hat{D}(x)| \mathcal{A} (\theta_t x)) = 0$, $\bP$-a.s.

\begin{corollary}\label{col4.6}
Assume that the conditions in Theorem \ref{attractor} are satisfied and, in addition, $f(0) = g(0)=0$. Then there exists $\epsilon>0$ such that for $C_g < \epsilon$ the random pullback attractor of the system \eqref{SDDE} provided by Theorem \ref{attractor} is the set $\mathcal{A}(x) = \{0\}$ which is both the random pullback and random forward attractor of the system \eqref{SDDE}.
\end{corollary}
\begin{proof}
Clearly the origin is a fixed point of the system \eqref{SDDE}, hence an invariant compact random invariant set of the system \eqref{SDDE}. By \eqref{y1} and the assumptions of Theorem \ref{attractor},  any solution of \eqref{SDDE} tends to the origin exponentially in the pullback sense, hence the set $\mathcal{A}(x) = \{0\}$ is the random pullback attractor of \eqref{SDDE} provided by Theorem \ref{attractor}. Similarly, by \eqref{yinfbetan3}  and the assumptions of Theorem \ref{attractor},  any solution of \eqref{SDDE} tends to the origin exponentially in the forward sense, hence $\mathcal{A}(x) = \{0\}$ attracts tempered compact random sets in the forward sense. Thus $\mathcal{A}(x) = \{0\}$ is also a random forward attractor of  \eqref{SDDE}.
\end{proof}
\medskip

The following theorem asserts that, in case $g$ is bounded, the existence of the random pullback attractor is ensured without further assumption on $C_g$. 

\begin{theorem}\label{thm.bound}
	Consider the system \eqref{SDDE}
\[
	dy(t) = [Ay(t)+ f(y_t)]dt + g(y_t)dZ(t),\quad y_0=\eta\in \cC^{0,\beta_0}([-r,0],\R^d).
\]	
Assume that the conditions  $\bf {H_1, H_2,H_3}$ hold, and additionally,
	\[
	C_AC_f<\lambda e^{-\lambda r}.
	\] 
	Assume furthermore that $g$ is bounded, i.e.\ $\sup_{\eta\in \cC^{0,\beta_0}([-r,0],\R^d)} \|g(\eta)\| <\infty$. Then  the generated random dynamical system of \eqref{SDDE} possesses a random pullback attractor $\mathcal{A}(x)$ which is in $\cC^{\beta}([-r,0],\R^d) \subset \cC^{0,\beta_0}([-r,0],\R^d)$.
\end{theorem}
\begin{proof}
	First notice that this theorem does not assume the smallness of $C_g$, hence we need to employ the boundedness of $g$ instead to prove the existence of the random pullback attractor of \eqref{SDDE}. We will make some significant modification of the proof of Theorem \ref{attractor} here. 
	
	Recall from the proof of Theorem \ref{attractor} that the deterministic equation \eqref{YDDE}
\[
	dy(t) = [Ay(t)+ f(y_t)]dt + g(y_t)dx(t),\quad y_0=\eta\in \cC^{0,\beta_0}([-r,0],\R^d),
\]	
 is a representative path-wise equation of  the stochastic equation \eqref{SDDE}, and  with a little abuse of notation we will denote by $y(\cdot)$ both the solution to \eqref{SDDE} and the solution to \eqref{YDDE}.

	Put $\|g\|_\infty := \sup_{\eta\in \cC^{0,\beta_0}([-r,0],\R^d)} \|g(\eta)\| <\infty$.  We fix $\bar{r} =k_0r$, where $k_0\in \N$ will be chosen later.  
	Let $\mu_t$ be the solution of the ordinary differential equation
	\begin{equation}\label{deterministic.equa}
	d\mu(t) = [A\mu(t) + f(\mu_t)] dt, \quad t\geq 0,
	\end{equation}
	with the initial condition $\mu (t) = y(t), \; t\in [0,r]$.  By applying estimate \eqref{nt2003} to system \eqref{deterministic.equa}, we obtain for all $t\geq r$ 
	\begin{eqnarray}\label{esti.mu}
	\|\mu\|_{\infty,[t-r,t]} &\leq&  C_A e^{\lambda r}  e^{-\lambda_0  (t-r)}  \|\mu\|_{\infty,[0,r]}+ \frac{C_A \|f(0)\| }{\lambda_0 }.
	\end{eqnarray}	 
	This implies that
	\begin{eqnarray*}
		\|\mu\|_{\infty,[t-r,t]} &\leq&  C_A e^{\lambda r} \|\mu\|_{\infty,[0,r]}+ \frac{C_A \|f(0)\| }{\lambda_0 },\quad t\geq r.
	\end{eqnarray*}	 
	Therefore, for all $2r\leq s <t$ and $v\in [-r,0]$,
	\begin{eqnarray*}
		\|\mu_t(v) -\mu_s(v) \|&\leq&\int_{s+v}^{t+v} (L_f\|\mu_u\| +\|f(0)\|) du = 
		\int_{s+v}^{t+v} (L_f\sup_{-r\leq m\leq 0}\|\mu(u+m)\| +\|f(0)\|) du\\
		&\leq& \int_{s+v}^{t+v} \left[L_f \left(C_A e^{\lambda r}   \|\mu\|_{\infty,[0,r]}+ \frac{C_A \|f(0)\| }{\lambda_0 }\right)+\|f(0)\|\right] du\\
		&=& (t-s) \left(C_AL_fe^{\lambda r} \|\mu\|_{\infty,[0,r]} +C_AL_f \frac{\|f(0)\| }{\lambda_0 }+\|f(0)\|\right)\\
		&\leq& (t-s) (C_AL_fe^{\lambda r} + \lambda_0)\left( \|\mu\|_{\infty,[0,r]} + \frac{\|f(0)\| }{\lambda_0 }\right).
	\end{eqnarray*}	 	
	Consequently, for all $2r\leq s <t$,
	\begin{equation}\label{mu1}
	\|\mu_t-\mu_s\|=\|\mu_t(\cdot)-\mu_s(\cdot)\|_{\infty,[-r,0]} \leq  (t-s) (C_AL_fe^{\lambda r} + \lambda_0)\left( \|\mu\|_{\infty,[0,r]} + \frac{\|f(0)\| }{\lambda_0 }\right).
	\end{equation}
Assign $h(t) := y(t) - \mu(t)$, then $h$ satisfies the equation 
	\[
	dh(t) = [Ah(t) + f(y_t)-f(\mu_t)] dt + g(y_t) dx(t).
	\]
By asumption $\bf H_2$ and \eqref{yloeve}, it follows that for  all $2r\leq s<t$,
	\begin{eqnarray}\label{h1}
	\|h(t)-h(s)\|&\leq & \int_s^tL_f\|h_u\|du  +(t-s)^\nu\ltn x\rtn_{\nu,[s,t]}\left[\|g\|_{\infty}  + K(t-s)^\beta \ltn g(y_{\cdot})\rtn_{\beta,[s,t]}\right].
	\end{eqnarray}
On the other hand, 
\[
	\|g(y_u)-g(y_v)\|\leq  C_g\|y_{u} -y_{v}\|
	\leq  C_g \left(\|h_{u} -h_{v}\|+\|\mu_{u} -\mu_{v}\| \right).
\]	
Observe that if $\|\mu_{u} -\mu_{v}\|\geq 1$ then  $\|g(y_u)-g(y_v)\| \leq  
	2\|g\|_{\infty}\|\mu_{u} -\mu_{v}\|^\beta$, whereas if $\|\mu_{u} -\mu_{v}\| < 1$ then 
	$\|\mu_{u} -\mu_{v}\|  \leq \|\mu_{u} -\mu_{v}\| ^\beta$. That leads to the estimate
\[
	\|g(y_u)-g(y_v)\| \leq C_g \|h_{u} -h_{v}\|+ (2\|g\|_{\infty}\vee C_g)\|\mu_{u} -\mu_{v}\|^\beta.
\]	
Together with \eqref{mu1}, we can estimate the H\"older norm of $g(y_\cdot)$ as follows
	\[
	\ltn g(y_{\cdot})\rtn_{\beta,[s,t]}\leq C_g\ltn h\rtn_{\beta,[s-r,t]} + (2\|g\|_{\infty}\vee C_g)(C_AL_fe^{\lambda r} + \lambda_0)^\beta\left( \|\mu\|^\beta_{\infty,[0,r]} + \frac{ \|f(0)\|^\beta }{\lambda^\beta_0 }\right).
	\]
	Therefore \eqref{h1} leads to
	\begin{eqnarray}\label{h2}
	&&\| h(t)-h(s)\|\notag\\
	&&\leq  \int_s^tL_f\|h_u\|du  +(t-s)^\nu\ltn x\rtn_{\nu,[s,t]}\left[L_1+L_2\|\mu\|^\beta_{\infty,[0,r]} + KC_g(t-s)^\beta \ltn h\rtn_{\beta,[s-r,t]}  \right], 	\end{eqnarray}
	for all $kr\leq s<t \leq (k+1)r$, $k\geq 2$, with $L_1= \|g\|_{\infty}+Kr^\beta(2\|g\|_{\infty}\vee C_g)(C_AL_fe^{\lambda r} + \lambda_0)^\beta\frac{ \|f(0)\|^\beta }{\lambda^\beta_0 },\;\; L_2=   Kr^\beta(2\|g\|_{\infty}\vee C_g)(C_AL_fe^{\lambda r} + \lambda_0)^\beta$.
	Note that \eqref{h2} has the form of \eqref{ytsr} but somehow simpler (we may look at \eqref{ytsr} with $\|g(0)\|$ replaced by $L_1+L_2\|\mu\|^\beta_{\infty,[0,r]}$, and $f(0)$ and two further items in \eqref{ytsr} replaced by 0).
	By repeating the arguments in Proposition \ref{esty1} on the interval $\Delta_{k}$, we obtain a similar estimate to \eqref{ykk}, namely
	\begin{eqnarray}
	\|h\|_{\beta,\Delta_{k}}&\leq &\exp\left \{ 4L_fr+ (4L_fr\vee 2) \left[1+(2KC_gr^\nu)^{\frac{1}{\nu-\beta} } \ltn  x\rtn^{\frac{1}{\nu-\beta}}_{\nu,\Delta_{k}}\right] \right\}\notag\\
	&&\times \left( \|h\|_{\beta,\Delta_{k-1}} + \frac{L_1+L_2 \|\mu\|^\beta_{\infty,[0,r]} }{C_g}\right) - \frac{L_1+L_2 \|\mu\|^\beta_{\infty,[0,r]} }{C_g}\notag\\
	&\leq &\exp\left \{ (4L_fr\vee 2) \left[2+(2KC_gr^\nu)^{\frac{1}{\nu-\beta} } \ltn  x\rtn^{\frac{1}{\nu-\beta}}_{\nu,\Delta_{k}}\right] \right\}\notag\\
	&&\times \left( \|h\|_{\beta,\Delta_{k-1}} + \frac{L_1+L_2 \|\mu\|^\beta_{\infty,[0,r]} }{C_g}\right) - \frac{L_1+L_2 \|\mu\|^\beta_{\infty,[0,r]} }{C_g}.
	\end{eqnarray}
	By induction, we can prove that, for all $k\geq 2$,
	\begin{eqnarray} \label{estimate-h}
	\|h\|_{\beta,\Delta_{k}}
	&\leq &\exp\left \{ 2(k-1)(4L_fr\vee 2) +(4L_fr\vee 2) (2KC_gr^\nu)^{\frac{1}{\nu-\beta} } \sum_{i=2}^k \ltn  x\rtn^{\frac{1}{\nu-\beta}}_{\nu,\Delta_i} \right\}\notag\\
	&&\times \left( \|h\|_{\beta,[r,2r]} + \frac{L_1+L_2 \|\mu\|^\beta_{\infty,[0,r]} }{C_g}\right) - \frac{L_1+L_2 \|\mu\|^\beta_{\infty,[0,r]} }{C_g}.
	\end{eqnarray}
	Now, in a similar manner as in the proof of Proposition \ref{esty1} we can estimate $\|h\|_{\beta,[r,2r]}$ as follow 
	\[
	\|h\|_{\beta,[r,2r]} \leq D \Big(1+\|h\|_{\beta,[0,r]} \Big)  e^{D\ltn  x\rtn^{\frac{1}{\nu-\beta}}_{\nu,[r,2r]} },
	\]
	where $D$ is some positive constant independent of $k$ and $x$. Since $ \|h\|_{\beta,[0,r]}=0$ we can write \eqref{estimate-h} in the form
	\begin{eqnarray}\label{estimate-h1}
	\|h\|_{\beta,\Delta_{k}}
	&\leq &\|\mu\|^\beta_{\infty,[0,r]} \xi_1(k,x) +\xi_2(k,x),\quad \forall  k\geq 2,
	\end{eqnarray}
	in which  $\xi_1,\xi_2$ have form $\exp\left\{D\left(k+ \sum_{i=1}^k \ltn  x\rtn^{\frac{1}{\nu-\beta}}_{\nu,\Delta_i}\right) \right\}$ for some generic positive constant $D$ independent of $k$  and $x$. 
	Applying Young's inequality $ab\leq\beta a^{\frac{1}{\beta}}+(1-\beta)b^{\frac{1}{1-\beta}}$, $\forall a,b\geq 0$, to \eqref{estimate-h1} we can finally show that
	\begin{eqnarray}\label{h3}
	\|h\|_{\beta,\Delta_{k}}
	&\leq &\epsilon \beta\|\mu\|_{\infty,[0,r]} + \left(\frac{ \xi_1(k,x)}{\epsilon} \right)^{\frac{1}{1-\beta}}+ \xi_2(k,x),\quad \forall k \geq 2,\; \epsilon >0,
	\end{eqnarray}
	where we choose and fix $\epsilon >0$ small enough such that $\epsilon\beta <1/2$.
	
	Next,   to estimate $\|\mu\|_{\beta,\Delta_k}$ we use the argument as in \eqref{yadd2}, and by virtue of \eqref{esti.mu} 
	we obtain for all $kr \leq s < t\leq (k+1)r$
	\begin{eqnarray*}
		r^\beta\frac{\|\mu(t)-\mu(s)\|}{(t-s)^\beta} &\leq &\|f(0)\|r+L_f r\|\mu\|_{\infty,[s-r,s]}+ L_f\int_s^t r^\beta\ltn \mu\rtn_{\beta,[s,u]}du \notag\\
		&\leq &\|f(0)\|r+C_AL_fe^{\lambda r} re^{-\lambda_0 (k-1)r}\|\mu\|_{\infty,[0,r]}+ L_f\int_s^t r^\beta\ltn \mu\rtn_{\beta,[s,u]}du. \notag
	\end{eqnarray*}
	Again, the Gronwall Lemma~\ref{contgronwall} applied to the function $r^\beta\ltn \mu\rtn_{\beta,[kr,\cdot]}$, and similar arguments to the proof of  Lemma~\ref{est.ybeta} help to show that
	\begin{eqnarray*}
		r^\beta\ltn \mu \rtn_{\beta,\Delta_k}
		&\leq &\|f(0)\|e^{L_fr}r+C_AL_fe^{(2\lambda +L_f)r} re^{-\lambda_0 kr}\|\mu\|_{\infty,[0,r]}.
	\end{eqnarray*}
	Together with \eqref{esti.mu} this yields
	\begin{eqnarray}\label{mu3}
	\| \mu \|_{\beta,\Delta_k}
	&\leq &D\|f(0)\|+De^{-\lambda_0 kr}\|\mu\|_{\infty,[0,r]},
	\end{eqnarray}
	where $D$ is a positive constant independent of $k$ and $x$. It then follows from \eqref{h3} and \eqref{mu3} that
	\begin{eqnarray}\label{y3}
	\| y \|_{\beta,\Delta_k}&\leq & \| h \|_{\beta,\Delta_k}+ \| \mu \|_{\beta,\Delta_k}\notag\\
	&\leq& (\epsilon \beta +De^{-\lambda_0 kr})\|\mu\|_{\infty,[0,r]} + \xi(k,x)\notag\\
	&\leq& (\epsilon \beta +De^{-\lambda_0 kr})\|y\|_{\beta,[0,r]} + \xi(k,x),
	\end{eqnarray}
	where $\xi(k,x)$ has the  form similar to that of $\xi_1,\xi_2$ above.
	We choose and fix $k_0$ large enough so that $De^{-\lambda_0 k_0r} <1/2$, then $(\epsilon \beta +De^{-\lambda_0 k_0r}) =: \gamma <1$ by the choice of $\epsilon$ and 
	\begin{eqnarray}\label{y6}
	\| y \|_{\beta,\Delta_{k_0}}&\leq & \gamma \|y\|_{\beta,[0,r]} + \xi(k_0,x).
	\end{eqnarray}
	Consequently, since our equation is autonomous we can apply the above arguments to the shifted equation and get for all $n\in\N$, $n\geq 2$, 
	\begin{eqnarray}\label{y4}
	\| y \|_{\beta,\Delta_{nk_0}}&\leq & \gamma^n \|y\|_{\beta,[0,r]} + \sum_{i=0}^{n-1} \gamma^i  \xi(k_0,\theta_{(n-i)k_0r}x).
	\end{eqnarray}
	Replacing $x$ by $\theta_{-nk_0r}x$ leads to
	\begin{eqnarray}\label{y5}
	\| y (\cdot,\theta_{-nk_0r}x,\eta)\|_{\beta,\Delta_{nk_0}}&\leq & \gamma^n \|y (\cdot,\theta_{-nk_0r}x,\eta)\|_{\beta,[0,r]} + \sum_{i=0}^{n-1} \gamma^i  \xi(k_0,\theta_{-ik_0r}x), \quad n\geq 2.
	\end{eqnarray}
	Note that $\xi$ is tempered under the assumption \eqref{Gamma}. Using the same arguments as that at the end of the proof of Theorem~\ref{attractor}, taking into account that $\gamma<1$, we can find a tempered random variable $\hat{b}(x)$ such that 
	for any tempered compact random set $\hat {D}(\cdot)\in \cD$, there exits an integer time moment $n(x,\hat{D})>0$ such that 
\[
	\| y (\cdot,\theta_{-nk_0r}x,\eta)\|_{\beta,\Delta_{nk_0}} \leq {\hat b}(x),
\]	
for all  $n\geq n(x,\hat{D})$ and all $\eta\in \hat {D}(\theta_{-nk_0r}x)$. Here we only estimate the norm of $y$ on $\Delta_{nk_0r}$,  $n\geq 1$, but it is easy to show the same estimate for the norm of $y$ on the interval $[t-r,t]$ for $t\geq k_0r$. 
	Thus we find a compact absorbing set 
\[
	\mathcal{B}(x) := \{\eta\in \cC^{0,\beta_0}([-r,0],\R^d)| \|\eta\|_{\beta,[-r,0]}\leq \hat{b}(x) \}
\]	
for the random dynamical system generated by the equation \eqref{SDDE}. Consequently, the random dynamical system generated by the equation \eqref{SDDE} possesses a random pullback attractor $\mathcal{A} (x) \subset \mathcal{B} (x)$ (see \cite[Theorem\ 3.5]{flandolibjorn}). Clearly, $\mathcal{A} (x) \subset \cC^{\beta}([-r,0],\R^d) \subset \cC^{0,\beta_0}([-r,0],\R^d)$.
\end{proof}

\begin{theorem}\label{lin.attractor}
Assume that the conditions in Theorem \ref{attractor} are satisfied and, in addition, $g$ is a linear form on $\cC_r$. Then there exists $\epsilon>0$ such that for $C_g < \epsilon$ the random pullback attractor of the system \eqref{SDDE} provided by Theorem \ref{attractor} is a singleton. Moreover, it is also a forward attractor.
\end{theorem}
\begin{proof}
Suppose that there exist two distinct points $a_1(x), a_2(x) \in \mathcal{A}(x)\subset \cC^{\beta}([-r,0],\R^d)$, where $\mathcal{A}(x)$ is the random pullback attractor provided by Theorem \ref{attractor}. We show that this will lead to a contradiction. 

Fix $n\in \N$, $n\geq 2$. Put $x^*:=\theta_{-nr}x$ and consider the equation
\begin{equation}\label{equ.star}
dy(t) = [Ay(t) + f(y_t)]dt + g(y_t)dx^*(t).
\end{equation}
By the invariance principle there exist two different points $b_1=b_1(x^*), b_2=b_2(x^*) \in \mathcal{A}(x^*)$ such that
\[
a_i(x) = y_{nr}(\cdot,x^*,b_i), \quad i = 1,2,
\]
where,  $y(\cdot,x^*,b_i)$ denotes the solution of \eqref{SDDE} with the driving path $x$ replaced by $x^*$ and the initial condition $\eta$ replaced by $b_i$, and 
$y_{nr}(\cdot,x^*,b_i)$ denotes the shifted function $y(\cdot+nr,x^*,b_i)$ considered as a function on $[-r,0]$.
Put $y^1(t) :=  y(t,x^*,b_1)$, $y^2(t) := y(t,x^*,b_2)$,  $ y(t):=y^1(t)-y^2(t) =  y(t,x^*,b_1)- y(t,x^*,b_2)$. Then $y^1(t) = y(t) + y^2(t)$, $y_{nr}(\cdot) =a_1(x)- a_2(x)$
and 
\begin{eqnarray}
	dy(t)&=& [Ay(t) + f(y_t+y^2_t)-f(y^2_t)]dt + [g(y^1_t)-g(y^2_t)]dx^*(t)\notag\\
	&=& [Ay(t) + f(y_t + y^2_t)-f(y^2_t)]dt + g(y_t)dx^*(t)\notag\\
	&=:&[Ay(t) + f^*(t,y_t)]dt + g(y_t)dx^*(t).\label{norm-difference-y}
\end{eqnarray}
Now we estimate the norm of $y(\cdot)$ using \eqref{norm-difference-y} and the method of the proof of Theorem \ref{attractor}. Notice that the results of  Theorem \ref{attractor} are not applicable directly to \eqref{norm-difference-y} because \eqref{norm-difference-y} is non-autonomous. However, a careful look at the proof of  Theorem \ref{attractor} assures us that, due to the specific construction of $f^*$ from $f$, this proof can be modified to the case of a non-autonomous system \eqref{norm-difference-y} as well to get some useful intermediate estimates.
Namely, taking into account that $g$ is a linear form, we repeat the calculation in the proof of  Theorem \ref{attractor}  in which $x$ is replaced by $x^*$, $f$ is replaced by the nonautonomous function $f^*$ (notice that the constants $C_f$, $L_f$ are not changed due to the construction of $f^*$ from $f$). Since $f^*(t,0)\equiv 0,\;g(0)=0$, similar to \eqref{yinfbetan3} we obtain 
	\begin{eqnarray}
	\|y\|_{\beta, \Delta_n}&\leq & M_8e^{-\lambda_0 nr} \|y^1-y^2\|_{\beta,[0,r]} \prod_{k=0}^{n-1} \Big[1+ C_gM_7G(\theta_{(k+1)r}x^*,[0,r])\Big]\notag\\
	&\leq &M_8e^{-\lambda_0 nr} \|y^1-y^2\|_{\beta,[0,r]} \prod_{k=0}^{n-1} \Big[1+ C_gM_7G(\theta_{-kr}x,[0,r])\Big].\label{eq44a}
	\end{eqnarray}
	Therefore, 
	\begin{eqnarray}
	\|a_1(x)- a_2(x)\|_{\beta,[-r,0]}&&=\quad   \|y_{nr}(\cdot)\|_{\beta,[-r,0]}=\|y\|_{\beta, \Delta_{n-1}}\notag\\
	&& \hspace{-2cm}\leq M_8e^{-\lambda_0 (n-1)r} \|y^1-y^2\|_{\beta,[0,r]} \prod_{k=0}^{n-1} \Big[1+ C_gM_7G(\theta_{-kr}x,[0,r])\Big].\label{normy}
	\end{eqnarray}
We estimate the terms in the right-hand side of the inequality in \eqref{normy}.	Using \eqref{ybeta0r} with $x$  replaced by $x^*$ and the fact that $\mathcal{A} \subset \mathcal{B}$, where $\mathcal{B}$ is determined at the end of the proof of Theorem \ref{attractor}, we obtain
\begin{eqnarray*}
	\|y^1-y^2\|_{\beta,[0,r]}  &\leq&\|y^1\|_{\beta,[0,r]}+\|y^2\|_{\beta,[0,r]} \\
	&\leq&  D\left(1  +\ltn x^*\rtn_{\nu,[0,r]} \right) \Big(2+\|y^1\|_{\beta_0,[-r,0]}+\|y^2\|_{\beta_0,[-r,0]} \Big)  e^{D \ltn  x^*\rtn^{\frac{1}{\nu-\beta_0}}_{\nu,[0,r]} }\\
	&\leq&  D\left(1  +\ltn x^*\rtn_{\nu,[0,r]} \right) \Big(2+\|b_1\|_{\beta_0,[-r,0]}+\|b_2\|_{\beta_0,[-r,0]} \Big)  e^{D \ltn  x^*\rtn^{\frac{1}{\nu-\beta_0}}_{\nu,[0,r]} }\\[2pt]
	&\leq& 2 (1+b(x^*))  \xi(\ltn x^*\rtn_{\nu,[0,r]})\\[2pt]
	&=& 2 (1+b(\theta_{-nr}x))  \xi(\ltn \theta_{-nr}x\rtn_{\nu,[0,r]}),
	\end{eqnarray*}
	 where $b(\cdot)$ is the diameter of  $\mathcal{B}$  which is tempered, and $\xi(\cdot)$ is a tempered function similar to that of the function ${\tilde F}(\cdot)$ in the proof of Theorem~\ref{attractor}.  By assumption of contradiction  $a_1(x) \not= a_2(x)$, thus $\|a_1(x)- a_2(x)\|_{\beta,[-r,0]}$ is a positive constant (for the fixed driving path $x$), which implies $\|a_1(x)- a_2(x)\|_{\beta,[-r,0]}>0$. Let $ n \to \infty$, it follows from \eqref{normy} that
	\begin{eqnarray*}
	&&0 = \varlimsup_{n\to\infty}\frac{1}{n}\log \|a_1(x)- a_2(x)\|_{\beta,[-r,0]} \notag\\
	&& \leq -\lambda_0 r  + \varlimsup_{n\to\infty}\frac{1}{n}\log  [ (1+ b(\theta_{-nr}x) ) \xi(\ltn \theta_{-nr}x\rtn_{\nu,[0,r]})]+ \varlimsup_{n\to\infty}\frac{1}{n} \sum_{k=0}^{n-1}\log \Big[1+ C_gM_7G(\theta_{-kr}x,[0,r])\Big]\notag\\
	&&\leq -\lambda_0 r + \hat{G}<0
	\end{eqnarray*}
	 if $C_g$ is small enough, where $\hat G$ is defined by \eqref{meanvalueG}. This contradiction proves that the pullback attractor $\mathcal{A}$ is a singleton. Taking into account \eqref{eq44a}, the arguments similar to that of Corollary~\ref{col4.6} in the forward direction of convergence show that $\mathcal{A}$ is also a forward attractor. 
\end{proof}

\section{Appendix}\label{appendix}
 
\subsection*{Young integrals}

For $[a,b]\subset\R$, denote by $\cC([a,b],\R^d)$ the space of all continuous functions $y: [a,b]\to \R^d$, equipped with the sup norm
\[\|y\|_{\infty,[a,b]} = \sup_{t\in [a,b]} \|y(t)\|,
\]
in which $\|\cdot \|$ is the Euclide norm of a vector in $\R^d$. 
Also, for $0<\beta \leq 1$ denote by $\cC^{\beta}([a,b],\R^d)$ the Banach space of all H\"older continuous paths $y:[a,b]\to \R^d$ with exponential $\beta$, equipped with the norm
\begin{eqnarray}
	\|y\|_{\infty,\beta,[a,b]} &:=& \|y\|_{\infty,[a,b]} + \ltn y\rtn_{\beta,[a,b]},\ \text{where}\notag\\
	\ltn y\rtn_{\beta,[a,b]}&:=&  \sup_{a\leq s<t\leq b} \frac{\|y(t)-y(s)\|}{(t-s)^{\beta}}< \infty.
\end{eqnarray}
One can easily prove for any $a\leq s\leq t\leq u\leq b$ that
\[
\ltn y\rtn_{\beta,[s,u]} \leq \ltn y\rtn_{\beta,[s,t]}+\ltn y\rtn_{\beta,[t,u]}.
\]
Note that the space $\cC^{\beta}([a,b],\R^d)$ is not separable. However, the closure of $\cC^{\infty}([a,b],\R^d)$ denoted by $\cC^{0,\beta}([a,b],\R^d)$ is a separable space (see \cite[Theorem 5.31, p. 96]{friz}), which can be defined as
\[
\cC^{0,\beta }([a,b],\R^d) := \Big\{x \in \cC^{\beta}([a,b],\R^d) \quad | \quad\lim \limits_{h \to 0} \quad \sup_{a\leq s<t\leq b, |t-s|\leq h} \frac{\|x(t)-x(s)\|}{(t-s)^{\beta}} = 0 \Big\}.
\]
It is worth to mention that for $\beta<\alpha$, $\cC^{\alpha}([a,b],\R^d)$ is a subspace of $\cC^{0,\beta}([a,b],\R^d)$ and moreover, the embedding operator 
\[
id: \cC^{\alpha}([a,b],\R^d) \rightarrow \cC^{\beta}([a,b],\R^d)
\]
 is compact  (see \cite[Proposition 5.28, \ p. 94]{friz}).\\
Now we recall that for $y \in \cC^{\beta}([a,b],\R^{d\times k})$ and $x \in \cC^{\nu}([a,b],\R^k)$ with $\beta+\nu>1$. Then the Young integral $\int_a^by(t)dx(t)$ exists  (see \cite{young}) and satisfies the Young-Loeve estimate \cite[Theorem 6.8, p.~116]{friz}, 
\[
	\left\|\int_s^ty(u)dx(u) - y(s)[x(t)-x(s)]\right\|
	\leq K (t-s)^{\beta+\nu}\ltn x \rtn_{\nu,[s,t]} \ltn y \rtn_{\beta,[s,t]},\quad \forall a \leq s\leq t \leq b,
\]
where $K:= \frac{1}{1-2^{1-(\beta+\nu)}}$. Hence
\begin{equation}\label{yloeve}
\left\|\int_s^ty(u)dx(u)\right\| \leq (t-s)^\nu\ltn x\rtn_{\nu,[s,t]} \left(\|y(s)\|+K(t-s)^{\beta}\ltn y\rtn_{\beta,[s,t]} \right).
\end{equation}

\subsection*{Tempered variables}
Let $(\Omega,\mathcal{F},\mP)$ be a probability space equipped with an ergodic metric dynamical system $\theta$, which is a $\mP$ measurable mapping $\theta: \Bbb T \times \Omega \to \Omega$, $\mathbb T$ is either $\R$ or $\Z$, and $\theta_{t+s} = \theta_t \circ \theta_s$ for all $t,s \in \Bbb T$. Recall that  a random variable $\rho: \Omega \to [0,\infty)$ is called {\em tempered} if

\begin{equation}\label{tempered}
\lim \limits_{t \to \pm \infty} \frac{1}{t} \log^+ \rho(\theta_{t}x) =0,\quad \text{a.s.}
\end{equation}
which, as shown in \cite[p. 220]{ImkSchm01}, \cite{GAKLBSch2010}, is equivalent to the sub-exponential growth
\[
\lim \limits_{t \to \pm \infty} e^{-c |t|} \rho(\theta_{t}x) =0\quad \text{a.s.}\quad \forall c >0.
\]
Note that our definition of temperedness corresponds to the notion of {\em temperedness from above} given in \cite[Definition 4.1.1(ii)]{arnold}. 
\begin{lemma}\label{lem.tempered1}
(i) If $h_1, h_2\geq 0$ are tempered random variables then $h_1+h_2$ and $h_1h_2$ are tempered random variables.\\
(ii) If $h_1\geq 0$ is a tempered random variable, $h_2\geq 0$ is a measurable random variable and $h_2\leq h_1$ almost surely, then $h_2$ is a tempered random variable.\\
(iii) Let $h_1$ be a nonnegative measurable function. If $\log^+ h_1\in L^1$ then $h_1$ is tempered.
\end{lemma}
\begin{proof}
$(i)$ See \cite[Lemma 4.1.2, p.~164]{arnold}.

$(ii)$ Immediate from the definition of tempered random variable, formula \eqref{tempered}.

$(iii)$ See \cite[Proposition 4.1.3, p.~165]{arnold}.
\end{proof}

\begin{lemma}\label{lem.tempered2}
$(i)$ Let $a : \Omega \to [0,\infty)$ be a random variable, $\log(1+a(\cdot))\in L^1$ and $\hat{a} :=  E\log(1+a(\cdot)) = \int_\Omega \log(1+a(\cdot)) d\mP$. Let $\lambda > \hat{a}$ be an arbitrary fixed positive number. Put
\[
b(x) := \sum_{k=1}^\infty e^{-\lambda k} \prod_{i=0}^{k-1} (1+ a(\theta_{-i} x)).
\]
Then $b(\cdot)$ is a nonnegative almost everywhere finite and tempered random variable.\\
$(ii)$ Let $c : \Omega \to [0,\infty)$ be a tempered random variable, and $\delta >0$ be an arbitrary fixed positive number. Put
\[
d(x) := \sum_{k=1}^\infty e^{-\delta k}  c(\theta_{-k} x).
\]
Then $d(\cdot)$ is a nonnegative almost everywhere finite and tempered random variable.
\end{lemma}
\begin{proof}
$(i)$ Put $b_n(x) := \sum_{k=1}^n e^{-\lambda k} \prod_{i=0}^{k-1} (1+ a(\theta_{-i} x))$.
Then $b_n(\cdot)$, $n\in \N$, is an increasing sequence of nonnegative random variable, hence converges to the nonnegative random variable $b(\cdot)$. Since $\log(1+a(\cdot))\in L^1$, by Birkhoff ergodic theorem there exists a $\theta$-invariant set $\Omega'\subset\Omega$ of full measure such that for all $x\in\Omega'$ we have
$\lim_{n\to\pm\infty} (\sum_{i=0}^{n-1}\log(1+ a(\theta_{-i} x))/n = \hat{a}$. Hence given any fixed $\delta>0$ for all $n$ big enough, 
$\prod_{i=0}^{n-1} (1+ a(\theta_{-i} x)) < \exp(\hat{a} +\delta)n$. Consequently, since $\lambda > \hat{a}$ the sequence  $b_n(\cdot)$, $n\in \N$ tends to  limit $b(\cdot)$, which is finite almost surely.

Now we show that $b(\cdot)$ is tempered. For $m\in\N$ and $x\in\Omega'$ we obtain
\begin{eqnarray*}
b(\theta_{-m} x) &=&
  \sum_{k=1}^\infty e^{-\lambda k} \prod_{i=0}^{k-1} (1+ a(\theta_{-i} \theta_{-m} x))
=  \sum_{k=1}^\infty e^{-\lambda k} \prod_{j=m}^{k+m-1} (1+ a(\theta_{-j} x))\\
&\leq & e^{\lambda m} \sum_{l=1+m}^\infty e^{-\lambda l} \prod_{j=0}^{l-1} (1+ a(\theta_{-j} x))
\leq  e^{\lambda m} \sum_{l=1}^\infty e^{-\lambda l} \prod_{j=0}^{l-1} (1+ a(\theta_{-j} x))
= e^{\lambda m} b(x).
\end{eqnarray*}
This implies that $\limsup_{m\to \infty} \frac{1}{m} \log^+ b(\theta_{-m} x) \leq  \lambda$.
 By virtue of
\cite[Proposition 4.1.3(i), p.~165]{arnold} and \cite[Lemma 4, Corollary 4]{Obrien82}, for all $x\in\Omega'$ we have 
\[
\limsup_{m\to \infty} \frac{1}{m} \log^+ b(\theta_{-m} x) = \limsup_{m\to -\infty} \frac{1}{-m} \log^+ b(\theta_{-m} x)= 0,
\] 
which proves that $b(\cdot)$ is tempered.

$(ii)$ Put $d_n(x) := \sum_{k=1}^n e^{-\delta k} c(\theta_{-k} x)$.
Then $d_n(\cdot)$, $n\in \N$, is an increasing sequence of nonnegative random variable, hence converges to the nonnegative random variable $d(\cdot)$.
By temperedness of  $c(\cdot)$ we can find a measurable set $\tilde\Omega\subset\Omega$ of full measure such that for all $x\in\tilde\Omega$ there exists  $n_0(x) >0$ such that for all $n\geq n_0(x)$ we have $c(\theta_{-n} x) \leq e^{n\delta/2}$. Hence $d_n(x)$, $n\in \N$, is an increasing sequence of positive numbers tending to finite value $d(x)$. Thus $d(\cdot)$ is finite almost everywhere. Furthermore,
for $m\in\N$ and $x\in\tilde\Omega$,
\[
d(\theta_{-m} x) =
  \sum_{k=1}^\infty e^{-\delta k} c(\theta_{-k} \theta_{-m} x)
=  \sum_{l=m+1}^\infty e^{-\delta (l-m)}  a(\theta_{-l} x)\\
\leq  e^{\delta m} \sum_{l=1}^\infty e^{-\delta l} c(\theta_{-l} x)
= e^{\delta m} d(x).
\]
This implies that $\limsup_{m\to \infty} \frac{1}{m} \log^+ d(\theta_{-m} x) \leq  \delta$.
Similar to (i) above, $d(\cdot)$ is tempered.
\end{proof}

\subsection*{Gronwall lemma}
\begin{lemma}[Continuous Gronwall Lemma]\label{contgronwall}
Let $[t_0,T]$ be an interval on $\R$.
	Assume that $u(\cdot), a(\cdot): [t_0,T] \rightarrow \R^+$ are positive continuous functions and $\beta >0$ is a positive number, such that
	\[
	u(t) \leq a(t) + \int_{t_0}^t \beta u(s) ds, \quad \forall t \in [t_0,T].
	\]
	Then the following inequality holds
	\[
	u(t) \leq a(t) + \int_{t_0}^t a(s) \beta e^{\beta (t-s)} ds, \quad \forall t \in [t_0,T].
	\]
\end{lemma}
\begin{proof}
	See \cite[Lemma 6.1, p\ 89]{amann}.
\end{proof}	
\begin{lemma}[Discrete Gronwall Lemma]\label{gronwall}
	Let $a$ be a non negative constant and $u_n, \alpha_n,\beta_n$ be  nonnegative sequences satisfying for all
	$n\in\N$, $n\geq 0$, the equalities 
	\[
	u_n\leq a + \sum_{k=0}^{n-1} \alpha_ku_k + \sum_{k=0}^{n-1} \beta_k.
	\]
	Then for all $n\in\N$, $n\geq 1$, the following inequalities hold
	\begin{equation}\label{estu}
	u_n\leq \max\{a,u_0\}\prod_{k=0}^{n-1} (1+\alpha_k) + \sum_{k=0}^{n-1}\beta_k\prod_{j=k+1}^{n-1}(1+\alpha_j).
	\end{equation}
\end{lemma}
\begin{proof} See \cite{duchong19}.
\end{proof}
\section*{Acknowledgments}
This work is supported by Vietnam National Foundation for Science and Technology Development (NAFOSTED) under grant number  101.03-2019.310. P.T. Hong would like to thank the IMU Breakout Graduate Fellowship Program for the financial support.

\end{document}